\documentclass[num-refs]{wiley-article}

\usepackage{siunitx}
\usepackage{calligra}
\usepackage{float}
\usepackage{graphicx}
\usepackage{subfigure}

\newtheorem{assum}{Assumption}
\usepackage{type1cm}

\newcounter{hl_thm}
\newtheorem{theo}[hl_thm]{Theorem}

\newlist{hip}{enumerate}{1}
\setlist[hip]{label=\textbf{(H.\arabic*)},labelindent=\parindent, leftmargin=*,nolistsep}
\setlist[hip]{label={(A.\arabic*)},labelindent=\parindent, leftmargin=*,nolistsep}

\papertype{PREPRINT SUBMITTED TO ARXIV.ORG}
\paperfield{2025}

\title{Gradient- and Newton-Based Unit Vector Extremum Seeking Control}

\author[1]{Roberto Luo}
\author[1]{Victor Hugo Pereira Rodrigues}
\author[1]{\\Tiago Roux Oliveira}
\author[2]{Miroslav Krstic}

\affil[1]{Department of Electronics and Telecommunication Engineering (DETEL), State University of Rio de Janeiro (UERJ), Rio de Janeiro -- RJ, Brazil.}
\affil[2]{University of California at San Diego (UCSD), San Diego -- CA, USA.}

\corraddress{Corresponding author: V. H. P. Rodrigues}
\corremail{victor.rodrigues@eng.uerj.br}

\fundinginfo{This research was supported by Coordenação de Aperfeiçoamento de Pessoal de Nível Superior (CAPES), Conselho Nacional de Desenvolvimento Científico e Tecnológico (CNPq), and Fundação de Amparo à Pesquisa do Estado do Rio de Janeiro (FAPERJ).}

\runningauthor{Roberto Luo et. al.}

\begin{document}

\begin{frontmatter}
\maketitle

\begin{abstract}
This paper presents novel methods for achieving stable and efficient convergence in multivariable extremum seeking control (ESC) using sliding mode techniques. Drawing inspiration from both classical sliding mode control and more recent developments in finite-time and fixed-time control, we propose a new framework that integrates these concepts into Gradient- and Newton-based ESC schemes based on sinusoidal perturbation signals. The key innovation lies in the use of discontinuous "relay-type" control components, replacing traditional proportional feedback to estimate the gradient of unknown quadratic nonlinear performance maps with Unit Vector Control (UVC). This represents the first attempt to address real-time, model-free optimization using sliding modes within the classical extremum seeking paradigm. In the Gradient-based approach, the convergence rate is influenced by the unknown Hessian of the objective function. In contrast, the Newton-based method overcomes this limitation by employing a dynamic estimator for the inverse of the Hessian, implemented via a Riccati equation filter. We establish finite-time convergence of the closed-loop average system to the extremum point for both methods by leveraging Lyapunov-based analysis and averaging theory tailored to systems with discontinuous right-hand sides. Numerical simulations validate the proposed method, illustrating significantly faster convergence and improved robustness compared to conventional ESC strategies, which typically guarantee only exponential stability. The results also demonstrate that the Gradient-based method exhibits slower convergence and higher transients since the gradient trajectory follows the curved and steepest-descent path, whereas the Newton-based method achieves faster convergence and improved overall performance going straightly to the extremum.
\end{abstract}
\end{frontmatter}

\section{Introduction}
 Extremum seeking control (ESC) is a real-time adaptive method designed to maintain the output of an objective function in the vicinity of its optimal value, whether to maximized or minimized, without requiring the knowledge of the system dynamics. The classical extremum seeking works by introducing a small periodic sinusoidal signal into the control input and output to estimate the gradient direction. As the estimate gradient approaches to zero, the output and input of the cost function converges to the vicinity of its extremum value \cite{KW:2000}. \\
\indent This control method has been employed in theoretical and application studies, such as in optimization field \cite{AK:2003, zhang2012extremum} for a wide class of problems: \textcolor{black}{in the presence of delays or dynamics modeled by Partial Differential Equations (PDEs) \cite{TRO_book2022}, for uncertain nonlinear systems subject to constraints \cite{GMD:2014} or underactuated systems with symmetry \cite{SK:2024}} as well as \textcolor{black}{for an iterative auto-tuning of the feedback gains for a class of nonlinear systems \cite{M:2016,TRO:2019,TRO:2020}}; in non-cooperative games to seek the Nash equilibrium \cite{TRoux:2021dec,ORKT:2021a}, and with event-triggered control \cite{ROHDK:2025,ROKB:2024a} in order to reduce control effort. All these works cited are based on a gradient method, the convergence rate speed is proportional to the unknown Hessian (second derivative) of the nonlinear map.\\ 
\indent Reference \cite{MMB:2010} has introduced a Newton-like method to estimate the second derivative of the map, to increase the convergence rate speed. Later, \cite{GKN:2012} proposed a dynamic estimator of the inverse of the Hessian in the form of a Riccati differential equation to eliminate the Hessian dependence and make it user-assignable. Therefore, the convergence rate speed of the system and the estimator of the Hessian inverse are independent of the unknown Hessian of the objective function. Due to this control strategy, it has been explored in extremum seeking field, for examples, in power optimization for photovoltaic to improve the transient performance \cite{GKS:2014}, the design of multivariable extremum seeking for static maps with different time delays \cite{OKT:2017}, to ensure uniform convergence rate of the system’s parameters for a multivariable static map \cite{GO:2024}, by using the static event-triggered control with a Newton-based extremum seeking \cite{ROKT:2025}, the controller updated less than the gradient-based method and for fast and accurate impedance-matching control \cite{GUO:2025}.\\
\indent In the literature on ESC based on sliding modes, a notable contribution is the work by Özgüner and coauthors \cite{Ozguner_SMC}, which introduced a periodic switching function to enable a state-feedback control approach. Building on this, Oliveira et al. \cite{Oliveira_SMC1,Oliveira_SMC2} extended the methodology to rely solely on output feedback. The same research group also proposed an alternative sliding mode-based ES controller employing monitoring functions \cite{Aminde_SMC1,Aminde_SMC2}, \textcolor{black}{including experimental validation by means of the source seeking application \cite{NOH:2024a}}. However, none of these ESC strategies adopt the original and simplest structure characterized by sinusoidal periodic perturbations, as introduced in \cite{KW:2000}. In our previous work \cite{ROKB:2024b}, we presented the first sliding-mode Nash equilibrium seeking (NES) strategy based on relay extremum seeking with sinusoidal perturbations. The approach featured two key characteristics in its distributed control architecture: (1) each player independently generated sinusoidal perturbations to estimate its pseudogradient without inter-agent communication, and (2) a discontinuous sliding-mode control law ensured finite-time convergence to the Nash equilibrium. The method combined extremum seeking principles with switching control actions to achieve model-free equilibrium seeking in quadratic duopoly games.\\
\indent Recently, interesting developments have emerged in the context of fixed-time strategies for ESC \cite{PK:2021,PKB:2023}---which are time-invariant and non-smooth—as well as prescribed-time schemes for extremum seeking \cite{YK:2022} and source seeking \cite{TK:2023}, which are time-varying and smooth. In particular, \cite{PK:2021,PKB:2023} introduce control laws reminiscent of those used in higher-order sliding mode control (SMC). A well-known example of second-order SMC is the super-twisting algorithm (STA) \cite{MO:2012}. More broadly, an entire family of finite-time (or fixed-time) SMC strategies exists \cite{M:2011}, offering potential for application in extremum seeking. Despite this, the idea of applying variable structure or sliding mode ESC strategies with finite (though not fixed) convergence time has yet to be fully explored within the SMC community. We believe this direction holds significant promise and, as a starting point, we focus on the simplest and most widely known case: first-order SMC (FOSMC), implemented via a relay function. Although FOSMC is discontinuous, it guarantees finite-time convergence. \\
\indent In this context, Unit-Vector Control (UVC) is a control strategy that normalizes the control input to unit magnitude, enforcing finite-time convergence of the system’s state trajectory through the application of a discontinuous control law \cite{B:1993, OFH:2023}. UVC it is a branch of Sliding Mode Control (SMC) tailored for multivariable systems, offering a distinct advantage in the design of robust controllers for systems subject to uncertainties. To bridge the existing gaps, our paper extends the classical multivaribale Gradient- and Newton-based ESC algorithms \cite{GKN:2012} by incorporating UVC in place of the usual proportional control laws with respect to the gradient estimates. We provide theoretical analysis, leveraging time-scaling techniques, constructing a Lyapunov function, and employing averaging methods to ensure closed-loop stability. Numerical simulations demonstrate the efficacy of our approaches, highlighting its superiority over classical ES with linear proportional control laws \cite{KW:2000,GKN:2012} since now the average closed-loop system is finite-time stable rather than only exponentially
stable. \\
%
\indent In the particular design of a Newton-based extremum seeking control strategy combined with Unit Vector Control (UVC), we enhance the convergence rate by eliminating the dependence on the unknown Hessian. A numerical example is provided to validate the theoretical findings, offering a comparative analysis between the Gradient- and Newton-based methods with finite-time convergence, both employing UVC.\\

\section{Gradient-based Unit Vector Extremum Seeking Control}

We define the following  nonlinear static map
\begin{align}
y(t)&=Q(\theta(t))\,, \label{eq:y_v1} \\
&= Q^{\ast}+\frac{1}{2}(\theta(t)-\theta^{\ast})^{\top}H^{\ast}(\theta(t)-\theta^{\ast})\,, \label{eq:y_v2} \\
&=Q^{\ast}+\frac{1}{2}\sum_{j=1}^{N}\sum_{k=1}^{N}H_{jk}^{\ast}(\theta_{j}(t)-\theta^{\ast}_{j})(\theta_{k}(t)-\theta^{\ast}_{k})\,, \label{eq:y_v3}
\end{align}
where $Q^{\ast}\in  {R}$ is the extremum point, $H^{\ast}=H^{\ast \top} \in  {R}^{n \times n}$ is the Hessian matrix, $\theta^{\ast} \in  {R}^{n}$ is the optimizer vector, $\theta(t)\in  {R}^{n}$ is the input map. The cost function $Q(\theta(t))$ in (\ref{eq:y_v1}) is not explicitly known, however, we have access to measurements of $y(t) \in  {R}$ and can adjust $\theta(t)$. The Gradient-based Extremum Seeking with unit vector control approach for this multivariable static map is illustrated in Fig.~\ref{fig:BD_ESC_UVC}. In this scheme, the feedback gain is
\begin{figure}[!ht]
\centering
\includegraphics[width=6.cm]{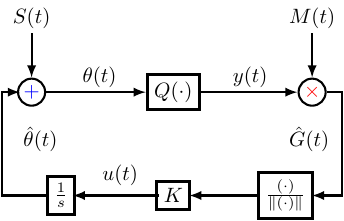}
\caption{Block diagram of extremum seeking with unit vector control.}
\label{fig:BD_ESC_UVC}
\end{figure}
\begin{align}
K=\text{diag}\left\{K_{1}\,,K_{2}\,,\ldots\,,K_{n}\right\}\,, \label{eq:K}
\end{align}
and the dither signals are defined as (see \cite{GKN:2012,K:2014})
\begin{align}
S(t)&= \left[a_{1}\sin\left(\omega_1 t\right),\ldots,a_{i}\sin\left(\omega_i t\right),\ldots,a_{n}\sin\left(\omega_n t\right)\right]^{\top}\!\!, \label{eq:S_v1} \\
M(t)&=2\left[\frac{\sin\left(\omega_1 t\right)}{a_{1}},\ldots,\frac{\sin\left(\omega_i t\right)}{a_{i}},\ldots,\frac{\sin\left(\omega_n t\right)}{a_{n}}\right]^{\top}, \label{eq:M_v1}
\end{align}
of nonzero amplitudes $a_{i}$. Finally, the probing frequencies $\omega_{i}$'s can be selected as
\begin{align}
\omega_{i}=\omega_{i}'\omega \,, \quad i \in \left\{1,\ldots\,,n\right\}\,, \label{eq:omegai_event}
\end{align}
where $\omega$ is a positive constant and $\omega_{i}'$ is a rational number.
\begin{assum}
\label{assumption_w}
The probing frequencies satisfy
\begin{align}
\omega'_{i} 	\notin \left\{\omega'_{j}\,,~\frac{1}{2}(\omega'_{j}+\omega'_{k})\,,~\omega'_{j}+2\omega'_{k}\,,~\omega'_{k}\pm\omega'_{l}\right\}\,, \label{eq:omega_iNotIn}
\end{align}
for all $i$, $j$, $k$ and $l$.
\end{assum} 
The block diagram in Fig.~\ref{fig:BD_ESC_UVC} highlights the main components of the ES architecture: the gradient estimation mechanism driven by periodic perturbations and demodulation technique and the control block adjusts the input vector using the estimated gradient through a normalization to unit magnitude.
Moreover, if Assumption~\ref{assumption_w} holds, an accurate estimate of the $i$-th component of the unknown gradient vector is given by $\hat{G}(t) = M(t)y(t) \in  {R}^{n}$ whose $i$-th component can be expressed as
\begin{align}
\hat{G}_{i}(t)&=\frac{2}{a_{i}}\sin(\omega_{i}t)y(t)\,. \label{eq:Gi}
\end{align}
On the other hand, the output of the integrator, the vector $\hat{\theta}(t)\in  {R}^{n}$, give us an estimate of $\theta^{*} \in  {R}^{n}$ such that the \textit{estimation error} is defined by 
\begin{align}
\tilde{\theta}(t)&=\hat{\theta}(t)-\theta^{*}\,.\label{eq:tildeThetai_v1}
\end{align}
 The input to the nonlinear map \( Q(\theta) \) in (\ref{eq:y_v1}), as described in Fig.~\ref{fig:BD_ESC_UVC}, is given by \( \theta(t) = \hat{\theta}(t) + S(t) \). For analysis purposes, and by using (\ref{eq:S_v1}) and (\ref{eq:tildeThetai_v1}), the \( i \)-th component of the input vector \( \theta(t) \) can be expressed in terms of the \( i \)-th component of the unknown estimation error vector \( \tilde{\theta}(t) \), the \( i \)-th component of the unknown optimal parameter vector \( \theta^{\ast} \), and the \( i \)-th component of the known dither signal \( S(t) \), as follows
\begin{align}
\theta_{i}(t)&=\tilde{\theta}_{i}(t)+a_{i}\sin\left(\omega_i t\right)+\theta_{i}^{*}\,.\label{eq:thetai_v1}
\end{align}
Therefore, by substituing (\ref{eq:y_v3}),  (\ref{eq:M_v1}) and (\ref{eq:tildeThetai_v1}) into (\ref{eq:Gi}), the $i$-th component of the estimate of the unknown gradient vector can be expressed as
\begin{align}
&\hat{G}_{i}(t)=\frac{1}{a_{i}}\sin(\omega_{i}t)\tilde{\theta}^{\top}(t)H^{\ast}\tilde{\theta}(t)+\sum_{j=1}^{n}\mbox{\calligra H}_{~~ij}^{~~~\ast}(t)\tilde{\theta}_{j}(t) +\Delta_{i}(t) \,, \label{eq:hatG_20250522_1} \\
&\mbox{\calligra H}_{~~ij}^{~~~\ast}(t):=H_{ij}^{\ast}+\Delta\mbox{\calligra H}_{~~ij}^{~~~\ast}(t)\,, \label{eq:calligraH_ij_20250522_1} \\
&\Delta\mbox{\calligra H}_{~~ij}^{~~~\ast}(t):=-H_{ij}^{\ast}\cos(2\omega_{i}t)+\sum_{\substack{k=1 \\ k\neq i}}^{n}H_{kj}^{\ast}\frac{a_{k}}{a_{i}}\cos((\omega_{i}-\omega_{k})t) -\sum_{\substack{k=1 \\ k\neq i}}^{n}H_{kj}^{\ast}\frac{a_{k}}{a_{i}}\cos((\omega_{i}+\omega_{k})t)\,,\label{eq:DeltaCalligraH_ij_20250522_1}\\
&\Delta_{i}(t):=\frac{2Q^{\ast}}{a_{i}}\sin(\omega_{i}t) + \frac{1}{4}\sum_{j=1}^{n}\sum_{k=1}^{n}H_{jk}^{\ast}\frac{a_{j}a_{k}}{a_{i}}\sin((\omega_{i}+\omega_{j}-\omega_{k})t)-\frac{1}{4}\sum_{j=1}^{n}\sum_{k=1}^{n}H_{jk}^{\ast}\frac{a_{j}a_{k}}{a_{i}}\sin((\omega_{i}-\omega_{j}+\omega_{k})t) \nonumber \\
&\qquad~~~~~~-\frac{1}{4}\sum_{j=1}^{n}\sum_{k=1}^{n}H_{jk}^{\ast}\frac{a_{j}a_{k}}{a_{i}}\sin((\omega_{i}+\omega_{j}+\omega_{k})t) -\frac{1}{4}\sum_{j=1}^{n}\sum_{k=1}^{n}H_{jk}^{\ast}\frac{a_{j}a_{k}}{a_{i}}\sin((\omega_{i}-\omega_{j}-\omega_{k})t)\,. \label{eq:Delta_i_20250222_1}
\end{align}
Thus, by using (\ref{eq:calligraH_ij_20250522_1})--(\ref{eq:Delta_i_20250222_1}), it is define the time-varying matrices $\mbox{\calligra H}^{~~~\ast}(t) \in  {R}^{n \times n}$ and $\Delta \! \mbox{\calligra H}^{~~~\ast}(t) \in  {R}^{n \times n}$, and the time-varying vector $\Delta(t) \in  {R}^{n }$,  
\begin{align}
\mbox{\calligra H}^{~~~\ast}(t) &:= H^{\ast}+\Delta \! \mbox{\calligra H}^{~~~\ast}(t)\,, \label{eq:calligraH}\\
\Delta \! \mbox{\calligra H}^{~~~\ast}(t)&:= \begin{bmatrix}
													\Delta \! \mbox{\calligra H}^{~~~\ast}_{~~~11}(t) & \Delta \! \mbox{\calligra H}^{~~~\ast}_{~~~12}(t) & \ldots & \Delta \! \mbox{\calligra H}^{~~~\ast}_{~~~1n}(t) \\
													\Delta \! \mbox{\calligra H}^{~~~\ast}_{~~~21}(t) & \Delta \! \mbox{\calligra H}^{~~~\ast}_{~~~22}(t) & \ldots & \Delta \! \mbox{\calligra H}^{~~~\ast}_{~~~2n}(t) \\
													\vdots                         & \vdots                         &        & \vdots                         \\
													\Delta \! \mbox{\calligra H}^{~~~\ast}_{~~~n1}(t) & \Delta \! \mbox{\calligra H}^{~~~\ast}_{~~~n2}(t) & \ldots & \Delta \! \mbox{\calligra H}^{~~~\ast}_{~~~nn}(t) \\
												 \end{bmatrix} \,, \label{eq:DeltaCalligraH} \\
						 \Delta(t) &:= \begin{bmatrix}
													\Delta_{1}(t) \,, 
													\Delta_{2}(t) \,,
													\ldots\,,
													\Delta_{n}(t)
												 \end{bmatrix}^{\top} \,. \label{eq:Delta}
\end{align}
Then, by using (\ref{eq:M_v1}) and (\ref{eq:calligraH})--(\ref{eq:Delta}), we can express (\ref{eq:hatG_20250522_1}) in the compact form
\begin{align}
\hat{G}(t)& {=}M(t)\frac{1}{2}\tilde{\theta}^{\top}(t)H^{\ast}\tilde{\theta}(t) {+}\left(H^{\ast} {+}\Delta \! \mbox{\calligra H}^{~~~\ast}\!\!(t)\right)\tilde{\theta}(t) {+}\Delta(t), \label{eq:hatG_20240302_2}
\end{align}
where $\Delta \! \mbox{\calligra H}~(t)$, defined in (\ref{eq:DeltaCalligraH}), and $\Delta(t)$, given in (\ref{eq:Delta}) are time-varying matrix and vector, respectively, both with zero mean.
Since the term $\tilde{\theta}^{\top}(t)H^{\ast}\tilde{\theta}(t)$ is quadratic in $\tilde{\theta}(t)$ and, therefore, may be neglected in a local analysis \cite{AK:2003},  the gradient estimate (\ref{eq:hatG_20240302_2}) ca be rewritten as 
\begin{align}
\hat{G}(t)&=\left(H^{\ast}+\Delta \! \mbox{\calligra H}^{~~~\ast}(t)\right)\tilde{\theta}(t)+\Delta(t)\,. \label{eq:hatG_20240302_3}
\end{align}
On the other hand, from the time-derivative of (\ref{eq:tildeThetai_v1}) and the ESC scheme depicted in Fig.~\ref{fig:BD_ESC_UVC}, the dynamics that governs $\hat{\theta}(t)$, as well as $\tilde{\theta}(t)$, is given by
\begin{align}
\frac{d\tilde{\theta}}{dt}(t)&=\frac{d\hat{\theta}}{dt}(t)=u(t) \label{eq:gradient_dtildeThetadt_20250206_1}\,, 
\end{align}
with $u(t)= [u_{1}(t)\,,u_{2}(t)\,,\ldots\,,u_{n}(t)]^{\top} \in  {R}^{n}$. Moreover, by using (\ref{eq:gradient_dtildeThetadt_20250206_1}), the time-derivative of (\ref{eq:hatG_20240302_3}) is given by 
\begin{align}
\frac{d\hat{G}}{dt}(t)& {=}\left(H^{\ast} {+}\Delta \! \mbox{\calligra H}^{~~~\ast}\!(t)\right)u(t) {+}\frac{d\Delta \! \mbox{\calligra H}^{~~~\ast}}{dt}\!(t)\tilde{\theta}(t) {+}\frac{d \Delta}{dt}(t), \label{eq:gradient_dhatGdt_20250206_1}
\end{align}
where, from (\ref{eq:calligraH}) and (\ref{eq:Delta}), it is easy to verify that $\dfrac{d\Delta \! \mbox{\calligra H}^{~~~\ast}}{dt}(t)$ as well as $\dfrac{d \Delta}{dt}(t)$ have null mean values.\\
\indent In Fig.~\ref{fig:BD_ESC_UVC} is depicted an equivalent closed-loop representation of a multivariable extremum seeking system, where a general control law $u(t)$ is driven by a gradient estimate. This formulation highlights the separation between estimation and control, enabling flexible design of optimization-based controllers and supporting theoretical analysis under standard assumptions.
\vspace{-0.25cm}
\subsection{Gradient-Based Unit Vector Control Law}
We assume the static state-feedback control law, for all $t\geq 0$, 
\begin{align}
u(t)=K~\frac{\hat{G}(t)}{\|\hat{G}(t)\|} \,, \quad \forall t\geq 0 \label{eq:u}
\end{align}
being such that $KH^{\ast}$ is Hurwitz. 

We assume that the control law (\ref{eq:u}) stabilizes the system at the corresponding equilibrium $\hat{G} \equiv 0$ with exponential convergence. This means that any trajectory of $\hat{\theta}(t)$ starting sufficiently near to the extremum point $\theta^{\ast}$ will converge to it exponentially, with uniform decay and overshoot bounds. It simply formalizes the idea that the control law is designed for local stabilization, regardless of the specific value of $\theta$. The control law (\ref{eq:u}) does not rely on detailed knowledge of the mapping (\ref{eq:y_v3}) or the system dynamics (\ref{eq:gradient_dtildeThetadt_20250206_1}) and (\ref{eq:gradient_dhatGdt_20250206_1}). Its design is model-independent in that sense, which is particularly advantageous in scenarios where precise models are unavailable or hard to obtain.

\subsection{Closed-loop System}

Plugging the proposed sliding mode based control law in
(\ref{eq:u}) into the time-derivative of $\hat{G}(t)$ and $\tilde{\theta}(t)$ yield the following dynamics governing equations,
\begin{align}
\frac{d\hat{G}(t)}{dt}=&\left(H^{\ast}+\Delta \! \mbox{\calligra H}^{~~~\ast}(t)\right)\ K\frac{\hat{G}(t)}{\|\hat{G}(t)\|}+\frac{d\mbox{\calligra H}~~(t)}{dt}\tilde{\theta}(t)+\frac{d \Delta(t)}{dt}\,, \label{eq:dhatGdt_20250206_2}\\
\frac{d\tilde{\theta}(t)}{dt} = & K\frac{\left(H^{\ast}+\Delta \! \mbox{\calligra H}^{~~~\ast}(t)\right)\tilde{\theta}(t)+\Delta(t)}{\left\|\left(H^{\ast}+\Delta \! \mbox{\calligra H}^{~~~\ast}(t)\right)\tilde{\theta}(t)+\Delta(t)\right\|}\,.
\label{eq:dtildeThetadt_20250206_2}
\end{align}

Now, it is possible to analyze the stability of (\ref{eq:dhatGdt_20250206_2}) and (\ref{eq:dtildeThetadt_20250206_2}) by using the average method. But before, it is necessary to write them in a appropriate form through a time scaling and augmented state to analyze the influense of $\omega$ in the dynamic system. 

The closed-loop system described by (\ref{eq:dhatGdt_20250206_2}) and (\ref{eq:dtildeThetadt_20250206_2}) highlights a crucial point: even if the product {\calligra H}~~(t)$K$ on averaging sense would form a Hurwitz matrix, the convergence to the equilibrium $\hat{G}(t) \equiv 0 $ and $\tilde{\theta}(t) \equiv 0$ would not be directly guaranteed through the $\tilde{\theta}$–dynamics. On the other hand, it is important to note that the time-varying disturbances $\Delta$ well as the time-varying matrix $\frac{d\mbox{\calligra H}(t)}{dt}$ possess zero mean values, which will be helpful in the following averaging analysis.

\subsection{Time Scaling}

To simplify the stability analysis of the closed-loop system, we introduce an appropriate time scale. By applying this time-scaling, we enable the use of averaging techniques to analyze the long-term behavior of  $\hat{G}(t)$ and $\tilde{\theta}(t)$. This approach focuses on the dynamics evolving on the slow timescale $\bar{t}$, while effectively filtering out the fast timevarying oscillations in the vector fields of (\ref{eq:dhatGdt_20250206_2}) and (\ref{eq:dtildeThetadt_20250206_2}).

By examining (\ref{eq:omega_iNotIn}), it is evident that the dither frequencies (\ref{eq:S_v1}) and (\ref{eq:M_v1}) as well as their combinations, are rational. Additionally, there exists a time period T such that
\begin{align}
T&:=2\pi \times \mbox{LMC} \left(\frac{1}{\omega_{1}}, \frac{1}{\omega_2}, ..., \frac{1}{\omega_n} \right), \label{eq:omega_event_1_siso}
\end{align}
where LCM denotes the least common multiple. Hence, we define a new time scale for the dynamics (\ref{eq:dhatGdt_20250206_2}) and (\ref{eq:dtildeThetadt_20250206_2}) by using the time scaling $\bar{t} = \omega t$,
\begin{align}
    \omega := \frac{2\pi}{T}\,.
    \label{eq:omega_scaled}
\end{align}
\indent Substituing in (\ref{eq:dhatGdt_20250206_2}) and (\ref{eq:dtildeThetadt_20250206_2}) result,
\begin{align}
\frac{d\hat{G}(\bar{t})}{d\bar{t}}& {=}\frac{1}{\omega}\left(H^{\ast}+\Delta \! \mbox{\calligra H}^{~~~\ast}(t)\right)K\frac{\hat{G}(t)}{\|\hat{G}(t)\|}  {+}\frac{1}{\omega}\frac{d\mbox{\calligra H}~~(t)}{dt}\tilde{\theta}(t) {+}\frac{1}{\omega}\frac{d \Delta(t)}{dt}\,, \label{eq:gradient_dotHatGav_event_4_siso} \\
\frac{d\tilde{\theta}(\bar{t})}{d\bar{t}}& {=} \frac{1}{\omega}K\frac{\left(H^{\ast}+\Delta \! \mbox{\calligra H}^{~~~\ast}(t)\right)\tilde{\theta}(t)+\Delta(t)}{\left\|\left(H^{\ast}+\Delta \! \mbox{\calligra H}^{~~~\ast}(t)\right)\tilde{\theta}(t)+\Delta(t)\right\|}\,. \label{eq:gradient_dotTildeTheta_3_event_siso}
\end{align}
\subsection{Average Closed-Loop System}
Now, by defining the augmented state
\textcolor{black}{
\begin{align}
X^{\top}(\bar{t}):=\begin{bmatrix}\hat{G}^{\top}(\bar{t})\,, \tilde{\theta}^{\top}(\bar{t}) 
\end{bmatrix}^{\top}\,, \label{eq:xT_vector}
\end{align}
}
one arrives at the dynamics
\textcolor{black}{
\begin{align}
\dfrac{dX(\bar{t})}{d\bar{t}}&=\dfrac{1}{\omega}\mathcal{F}\left(\bar{t},X,\dfrac{1}{\omega}\right) \label{eq:gradient_dotX_event_siso}
\,.
\end{align}}
\indent Due to the discontinuous nature of the proposed control strategy, the averaging theory for discontinuous systems is employed in the paper, according to \cite{P:1979}.\\
\indent The augmented system (\ref{eq:gradient_dotX_event_siso}) has a small parameter $1/\omega$ as well as a $T$-periodic function $\mathcal{F}\left(\bar{t},X,\dfrac{1}{\omega}\right)$ in $\bar{t}$, hence it can be studied through the averaging method for stability analysis by averaging  $\mathcal{F}\left(\bar{t},X,\dfrac{1}{\omega}\right)$ at $\displaystyle \lim_{\omega\to \infty}\dfrac{1}{\omega}=0$, as shown in \cite{P:1979}, {\it i.e.},
\begin{align}
\dfrac{dX_{\rm av}(\bar{t})}{d\bar{t}}&=\dfrac{1}{\omega}\mathcal{F}_{\rm av}\left(X_{\rm av}\right) \,, \label{eq:gradient_dotXav_event_1_siso} \\
\mathcal{F}_{\rm av}\left(X_{\rm av}\right)&=\dfrac{1}{T}\int_{0}^{T}\mathcal{F}\left(\delta,X_{\rm av},0\right)d\delta
\,.  \label{eq:gradient_mathcalFav_event_siso}
\end{align}
Basically, the problem in the averaging method is to determine in what sense the behavior of the autonomous system (\ref{eq:gradient_dotXav_event_1_siso}) approximates the behavior of the nonautonomous system (\ref{eq:gradient_dotX_event_siso}) such that (\ref{eq:gradient_dotX_event_siso}) can be represented as a perturbed version of the system (\ref{eq:gradient_dotXav_event_1_siso}).

Therefore, treating the non-periodic states $\hat{G}(\bar{t})$, $\tilde{\theta}(\bar{t})$ as constants in (\ref{eq:gradient_dotHatGav_event_4_siso})-(\ref{eq:gradient_dotX_event_siso}), one gets the following average closed-loop system 
\begin{align}
    \frac{d\hat{G}_{\rm av}(t)}{d\bar{t}} & = \frac{1}{\omega}H^{\ast}K\frac{\hat{G}_{\rm av}(\bar{t)}}{\|\hat{G}_{\rm av}(\bar{t)}\|}\,,\label{eq:gradient_dHatG_av}  \\
    \frac{d\tilde{\theta}_{\rm av}(\bar{t})}{d\bar{t}} & = \frac{1}{\omega}KH^{\ast}\frac{\tilde{\theta}_{\rm av}(\bar{t)}}{\|\tilde{\theta}_{\rm av}(\bar{t)}\|}\,,\label{eq:dTildeTheta_av}\\
    \hat{G}_{\rm av}(\bar{t})&=H^{\ast}\tilde{\theta}_{\rm av}(\bar{t})\,. \label{eq:gradient_HatG_av}
\end{align}

\subsection{Stability Analysis}
The next theorem guarantees finite-time stability of the proposed Slide Mode Control, as depicted in Fig. ~\ref{fig:BD_ESC_UVC}.
\begin{theo}
Consider the closed-loop average dynamics of the gradient estimate (\ref{eq:gradient_dHatG_av}) and that Assumption 1 holds. For a sufficiently large $\omega > 0$, defined in (\ref{eq:omega_scaled}), there exist a finite-time $t_{S} \in \left(0,\frac{2\omega\lambda_{max}(P)}{\lambda_{min}(Q)}\|\hat{G}_{\rm av}(0)\|\right]$, such
that the sliding motion on $\hat{G}_{\rm av}(t) \equiv 0$ occurs. 
Furthermore, for the non-average system (\ref{eq:gradient_dhatGdt_20250206_1}), the input converge to an error of order $\mathcal{O}(a + \frac{1}{\omega})$ 
\begin{align}
  \| \theta(t) - \theta^* \|& \leq
  \left\{
  \begin{array}{ll}
\|H^{\ast}\|\|H^{\ast-1}\|\sqrt{\dfrac{\lambda_{\max}(P)}{\lambda_{\min}(P)}} \|\theta(0) - \theta^{\ast} \| - \dfrac{1}{2} \dfrac{ \|H^{*-1}\|\lambda_{\min}(Q)}{\sqrt{\lambda_{\max}(P)} \sqrt{\lambda_{\min}(P)}} \, t+ \mathcal{O} \left( a + \frac{1}{\omega} \right) 
& \forall t \in [0, t_{S}), \\
 \mathcal{O} \left( a + \frac{1}{\omega} \right) , ~ &\forall t \in [t_{S}, +\infty).
  \end{array}
  \right. \label{eq:theta_minimal_v02}
\end{align}
And for the output,
\begin{align}
  |y(t) - Q^{\ast}| &\leq 
  \left\{
    \begin{array}{ll}
      2\|H^{\ast}\|\left(
      \|H^{\ast}\|\|H^{\ast-1}\| 
      \sqrt{\frac{\lambda_{\max}(P)}{\lambda_{\min}(P)}} 
      \| \theta(0) - \theta^{\ast} \| 
      - \frac{1}{2} 
      \frac{\|H^{\ast-1}\|\lambda_{\min}(Q)}{
        \sqrt{\lambda_{\max}(P)} \sqrt{\lambda_{\min}(P)}} 
      t \right)^{2} 
      + \mathcal{O}\left(a^{2} + \frac{1}{\omega^{2}}\right), 
      & \forall t \in [0, t_{S}), \\
      \mathcal{O}\left(a^{2} + \frac{1}{\omega^{2}}\right), 
      & \forall t \in [t_{S}, +\infty).
    \end{array}
  \right. \label{eq:gradient_y_norm_v03}
\end{align}
\end{theo}
\begin{proof}
From a Lyapunov function candidate,
\begin{align}
V_{\rm{av}}(\bar{t})=\hat{G}^{\top}_{\rm{av}}(\bar{t})P\hat{G}_{\rm{av}}(\bar{t}) \,, \quad P=P^T>0\,. \label{eq:lyapunov_function}
\end{align} 
the product $H^{*}K$ is Hurwitz, by a given $Q=Q^T>0$, exist $P=P^T$, a symmetric matrix, such that the Lyapunov's equation is $K^{\top}H^{*\top}P+PH^{*}K=-Q$. Thus, through a time scaling and a time derivative of (\ref{eq:lyapunov_function}) is
\begin{align}
\frac{dV_{\rm av}(\bar{t})}{d\bar{t}}= &~\frac{1}{\omega}\frac{\hat{G}_{\rm av}^{\top}(\bar{t)}}{\|\hat{G}_{\rm av}(\bar{t)}\|}K^{\top}H^{\ast T}P\hat{G}_{\rm{av}}(\bar{t}) + \frac{1}{\omega}\hat{G}^{\top}_{\rm av}(\bar{t)}P H^{\ast}K\frac{\hat{G}_{\rm av}(\bar{t)}}{\|\hat{G}_{\rm av}(\bar{t)}\|}\,,\nonumber \\
= & -\frac{1}{\omega}\frac{\hat{G}^{\top}_{\rm{av}}(\bar{t})Q\hat{G}_{\rm{av}}(\bar{t})}{\|\hat{G}_{\rm av}(t)\|}\,.\label{eq:lyapunov_derivative}
\end{align}
Through the Rayleigh-Ritz inequality \cite{K:2002}, 
\begin{align}
    \lambda_{\min}(Q)\|\hat{G}_{\rm av}(\bar{t})\|^{2}\leq V_{\rm av}(\bar{t}) \leq \lambda_{\max}(Q)\|\hat{G}_{\rm av}(\bar{t})\|^{2}\,,\label{rayleigh-ritz_Q}
\end{align}
the Lyapunov candidate derivative (\ref{eq:lyapunov_derivative}) is limited by
\begin{align}
    \frac{dV_{\rm av}(\bar{t})}{d\bar{t}} \leq -\frac{1}{\omega}\lambda_{\min}(Q)\|\hat{G}_{\rm av}(\bar{t})\|\,.
    \label{eq:lyapunov_derivative2}
\end{align}
Now, by using the Rayleigh-Ritz inequality \cite{K:2002} to the $P$ matrix, 
\begin{align}
    \lambda_{\min}(P)\|\hat{G}_{\rm av}(\bar{t})\|^{2}\leq V_{\rm av}(\bar{t}) \leq \lambda_{\max}(P)\|\hat{G}_{\rm av}(\bar{t})\|^{2}\,,\label{rayleigh-ritz_P}
\end{align}
the equation (\ref{eq:lyapunov_derivative2}) is rewritten to
\begin{align}
    \frac{dV_{\rm av}(\bar{t})}{d\bar{t}} \leq -\frac{1}{\omega}\frac{\lambda_{\min}(Q)}{\sqrt{\lambda_{\max}(P)}}\sqrt{V_{\rm av}(\bar{t})}\,.
    \label{eq:lyapunov_derivative3}
\end{align}
To solve this inequality, the Comparison Lemma \cite{K:2002} is used,
\begin{align}
    V_{\rm av}(t) &\leq \bar{V}_{\rm av}(\bar{t})\,,~\forall \bar{t} \geq 0\,,\label{eq:comparison_lemma}
\end{align}
results, 
\begin{align}
    \frac{d\bar{V}_{\rm av}(\bar{t})}{d\bar{t}} & = -\frac{1}{\omega}\frac{\lambda_{\min}(Q)}{\sqrt{\lambda_{\max}(P)}}\sqrt{\bar{V}_{\rm av}(\bar{t})}\,, \quad \bar{V}_{\rm av}(0)=V_{\rm av}(0)\,. \label{eq:gradient_barV_v1}
\end{align}
The solution of (\ref{eq:gradient_barV_v1}) with the initial condition $\bar{V}_{\rm av}(0)$ is given by
\begin{align}
    \sqrt{\bar{V}_{\rm av}(\bar{t})} = \sqrt{\bar{V}_{\rm av}(0)}  -\frac{1}{2\omega}\frac{\lambda_{\min}(Q)}{\sqrt{\lambda_{\max}(P)}}\bar{t}\,.
    \label{eq:Vav_solution}
\end{align}
By using (\ref{rayleigh-ritz_P}), the equation (\ref{eq:Vav_solution}) can be upper bounded by,
\begin{align}
    \sqrt{\lambda_{\min}(P)}\|\hat{G}_{\rm av}(\bar{t})\| \leq &
    \sqrt{\lambda_{\max}(P)}\|\hat{G}_{\rm av}(0)\|
    -\frac{1}{2\omega}\frac{\lambda_{\min}(Q)}{\sqrt{\lambda_{\max}(P)}}\bar{t}\,,
\end{align}
thus,
\begin{align}
    \|\hat{G}_{\rm av}(\bar{t})\| \leq & \sqrt{\frac{\lambda_{max}(P)}{\lambda_{min}(P)}}\|\hat{G}_{\rm av}(0)\| - \frac{1}{2\omega}\frac{\lambda_{\min}(Q)}{\sqrt{\lambda_{\max}(P)}\sqrt{\lambda_{\min}(P)}}\bar{t}\,. \label{eq:Gav_solution}
\end{align}
From (\ref{eq:Gav_solution}), there exist a finite-time $t_{S} \in  \left(0,\frac{2\omega\lambda_{max}(P)}{\lambda_{min}(Q)}\|\hat{G}_{\rm av}(0)\|\right]$, such that the sliding motion occurs. Moreover, $H^{\ast}$ is invertible and, from (\ref{eq:gradient_HatG_av}), $\tilde{\theta}_{\rm av}(\bar{t}) = H^{\ast-1}\hat{G}_{\rm av}(\bar{t})$, the following inequality is established $\|H^{\ast-1}\|^{-1}\|\tilde{\theta}_{\rm av}(\bar{t})\| \leq \|H^{\ast}\tilde{\theta}_{\rm av}(\bar{t})\| \leq \|H^{\ast}\|\|\tilde{\theta}_{\rm av}(\bar{t})\|$ and, therefore, we can state that,
\begin{align}
\|\tilde{\theta}_{\rm av}(\bar{t})\| &\leq
\begin{cases}
\|H^{\ast}\|\|H^{\ast-1}\|\sqrt{\frac{\lambda_{\max}(P)}{\lambda_{\min}(P)}} \|\tilde{\theta}_{\rm av}(0)\| - \frac{1}{2\omega} \frac{\|H^{\ast-1}\|\lambda_{\min}(Q)}{\sqrt{\lambda_{\max}(P)} \sqrt{\lambda_{\min}(P)}} \bar{t}\,. \\
0, \quad \forall \bar{t} \geq \bar{t}_{S}.
\label{eq:gradient_norm_theta_zero}
\end{cases}
\end{align}

Since (\ref{eq:gradient_dhatGdt_20250206_1}) has a discontinuous right-hand side, but is also $T$-periodic in $t$, and noting that the average system with state 
$\tilde{\theta}_{\rm av}(\bar{t})$ is finite-time stable according to  (\ref{eq:gradient_norm_theta_zero}), we can invoke the averaging theorem in~\cite{P:1979} to conclude that
\begin{align}
    \|\tilde{\theta}(\bar{t}) - \tilde{\theta}_{\rm av}(\bar{t})\| \leq \mathcal{O}\left(\frac{1}{\omega}\right)\,. \label{eq:tilde_theta_theta_av}
\end{align}
By applying the triangle inequality \cite{A:1957}, we also obtain:
\begin{align}
\|\tilde{\theta}(\bar{t})\| 
&\leq \|\tilde{\theta}_{\rm av}(\bar{t})\| + \mathcal{O}\left( \frac{1}{\omega} \right)
\leq 
\|H^{\ast}\|\|H^{\ast-1}\|\sqrt{\dfrac{\lambda_{\max}(P)}{\lambda_{\min}(P)}} \|\tilde{\theta}_{\rm av}(0)\|
- \dfrac{1}{2\omega} \dfrac{\|H^{\ast-1}\| \lambda_{\min}(Q)}{\sqrt{\lambda_{\max}(P)} \sqrt{\lambda_{\min}(P)}} \, \bar{t}\,, 
\end{align}
Now, backing to the non-average system, from (\ref{eq:S_v1}) and from Fig. ~\ref{fig:BD_ESC_UVC}, we can verify that,
\begin{align}
    \theta(t) - \theta^{\ast} = \tilde{\theta}(t) + S(t)\,,
    \label{eq:theta_fig1}
\end{align}
whose Euclidian norm satisfies,
\begin{align}
    \| \theta(t) - \theta^* \| &= \| \tilde{\theta}(t) + S(t) \| \leq \| \tilde{\theta}(t) \| + \| S(t) \| \nonumber \\
    &\leq
    \begin{cases} 
     \|H^{*}\| \|H^{*-1}\| \sqrt{\dfrac{\lambda_{\max}(P)}{\lambda_{\min}(P)}} \|\theta(0) - \theta^{\ast}\|
- \dfrac{1}{2} \dfrac{\|H^{*-1}\| \lambda_{\min}(Q)}{\sqrt{\lambda_{\max}(P)} \sqrt{\lambda_{\min}(P)}} \, t + ~\mathcal{O}\left(a+ \dfrac{1}{\omega}\right), 
& \forall t \in [0, t_{S}),\\
         ~\mathcal{O}\left(a+ \dfrac{1}{\omega}\right), ~~ &\forall t \in [t_{S}, +\infty).
    \end{cases}
\label{eq:gradient_theta_diff}
\end{align}
From (\ref{eq:gradient_theta_diff}), the control strategy (\ref{eq:u}), influenced by the Hessian, $H^{\ast}$, allows the convergence to the residual set $\mathcal{O}\left(a + \frac{1}{\omega}\right)$ in a finite-time, within the interval  
\begin{align}
    t_{S} \in \left(0,~2\frac{\lambda_{max}(P)}{\lambda_{min}(Q)}\|H^{\ast}\| \|\theta(0) - \theta^{*}\|\right] \,.
\end{align}
From (\ref{eq:y_v2}) and the Cauchy-Schwarz inequality \cite{S:2008} it is possible to write,
\begin{align}
    |y(t) - Q^{*}|=|(\theta(t) - \theta^{\ast})^{T}H^{\ast}(\theta(t) - \theta^{\ast})|\leq \|H^{\ast}\|\|\theta(t) - \theta^{*}\|^{2}\,,
    \label{eq:gradient_norm_y}
\end{align}
and substituting (\ref{eq:gradient_theta_diff}) in (\ref{eq:gradient_norm_y}), lead to the following upper bound
\begin{align}
|y(t) - Q^{*}| 
& \leq \left\{
\begin{array}{ll}
\left[
     \|H^{\ast}\|\|H^{*-1}\|\sqrt{\frac{\lambda_{\max}(P)}{\lambda_{\min}(P)}} \| \theta(0) - \theta^* \| - \frac{1}{2} \frac{ \|H^{*-1}\|\lambda_{\min}(Q)}{
    \sqrt{\lambda_{\max}(P)} \sqrt{\lambda_{\min}(P)}} t 
    + \mathcal{O}\left(a + \frac{1}{\omega}\right)
\right]^2\,, &\forall t \in (0, t_{S}], \\
      \mathcal{O}\left(a+ \frac{1}{\omega}\right)^{2}, & \forall t \in (t_{S}, +\infty) .\label{eq:gradient_y_norm_v32} \\
\end{array}\right.
\\ & \leq \left\{
\begin{array}{ll}
2\|H^{\ast}\|\left(\|H^{\ast}\|\|H^{*-1}\|\sqrt{\frac{\lambda_{\max}(P)}{\lambda_{\min}(P)}} \| \theta(0) - \theta^* \| 
        - \frac{1}{2} \frac{\|H^{*-1}\|\lambda_{\min}(Q)}{
        \sqrt{\lambda_{\max}(P)} \sqrt{\lambda_{\min}(P)}} t \right)^{2} + \mathcal{O}\left(a^{2} + \frac{1}{\omega^{2}}\right),\hspace{-0.65cm} &\forall t \in (0, t_{S}], \\
      \mathcal{O}\left(a^{2} + \frac{1}{\omega^{2}}\right),\hspace{-0.65cm} & \forall t \in (t_{S}, +\infty) .\label{eq:gradient_y_norm_v3}
\end{array}\right.
\end{align}
From (\ref{eq:gradient_y_norm_v3}), the convergence of $|y(t) - Q^*|$ to the residual set $\mathcal{O}\left(a^2 + \frac{1}{\omega^{2}}\right)$ occurs for $t > t_{S}$. Then, the proof is completed.
\hfill$\square$
\end{proof}
\section{Newton-based Unit Vector Extremum Seeking Control}
\begin{figure}[H]
\centering
\includegraphics[width=8.2cm]{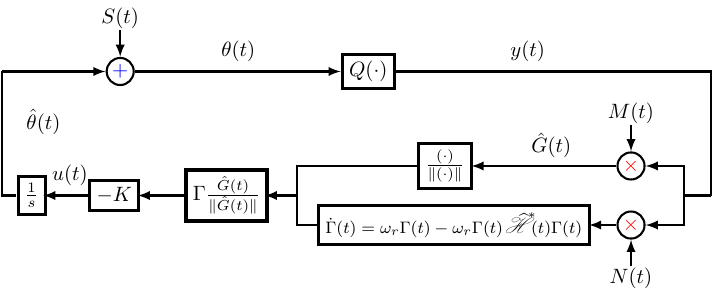}
\caption{Block diagram of Newton-based extremum seeking with unit vector control.}
\label{fig:BD_newton_ESC_UVC}
\end{figure}
Fig.~\ref{fig:BD_newton_ESC_UVC} depicts the Newton-based extremum seeking algorithm applied to a static map, where  $\omega_r$ is a positive real frequency used to drive periodic excitation. The algorithm relies on two essential mechanisms. The first is the \emph{perturbation structure}, defined by the time-varying matrix $N(t)\in  {R}^{n \times n}$, with elements 
\begin{align}
    N_{ii}(t)& {:=} {-}\frac{8}{a_{i}^2}\cos(2\omega_{i}t)\,, \label{eq:Nii}\\
    N_{ij}(t)& {:=}\frac{2}{a_{i}a_{j}}\cos((\omega_{i} {-}\omega_{j})t) {-}\frac{2}{a_{i}a_{j}}\cos((\omega_{i} {+}\omega_{j})t),~ j\neq i , \label{eq:Nij}
\end{align}
which introduces carefully chosen input oscillations that allow for the estimation of second-order information—namely, the Hessian matrix $H^{\ast}$ of the unknown cost function (\ref{eq:y_v2}), by means of
\begin{align}
    \widehat{\!\!\!\!\!\!\!\!\mbox{\calligra H}}^{~~~\ast}(t) &:= N(t)y(t)\,. \label{eq:calligraHHat_20250527}
\end{align}
Therefore, by plugging (\ref{eq:y_v3}),  (\ref{eq:Nii}) and (\ref{eq:Nij}) into (\ref{eq:calligraHHat_20250527}), the components of the estimate  ~~~~~$\widehat{\!\!\!\!\!\!\!\!\mbox{\calligra H}}^{~~\ast}_{~~ii}(t)$ and ~~~~~$\widehat{\!\!\!\!\!\!\!\!\mbox{\calligra H}}^{~~\ast}_{~~ij}(t)$ of the unknown Hessian matrix $H^{\ast}$ are given by
\begin{align}
    \widehat{\!\!\!\!\!\!\!\!\mbox{\calligra H}}^{~~\ast}_{~~ii}(t) &= H^{\ast}_{ii}  {+}\Delta~~~~\! \widehat{\mbox{\calligra \!\!\!\!\!\!\!\! H}}^{~~\ast}_{~~ii}(t) +N_{ii}(t)\left[\frac{1}{2}\tilde{\theta}^{\top}(t)H^{\ast}\tilde{\theta}(t) {+}S^{\top}(t)H^{\ast}\tilde{\theta}(t)\right]\,, \label{eq:calligraHiiHat} \\
    \Delta~~~~\! \widehat{\mbox{\calligra \!\!\!\!\!\!\!\! H}}^{~~\ast}_{~~ii}(t)&= -H^{\ast}_{ii}\cos(2\omega_{i}t)+N_{ii}(t)Q^{\ast} -2\sum_{\substack{k=1 \\ k\neq i}}^{n}H^{\ast}_{kk}\frac{a_{k}^{2}}{a_{i}^{2}}\cos(2\omega_{i}t)+\sum_{\substack{k=1 \\ k\neq i}}^{n}H^{\ast}_{kk}\frac{a_{k}^{2}}{a_{i}^{2}}\left[\cos(2(\omega_{i}-\omega_{k})t)+\cos(2(\omega_{i}+\omega_{k})t)\right] \nonumber  \\
    &-2\sum_{k=1}^{n}\sum_{l=k+1}^{n}H^{\ast}_{kl}\frac{a_{k}a_{l}}{a_{i}^2}\cos((2\omega_{i}+\omega_{k}-\omega_{l})t) -2\sum_{k=1}^{n}\sum_{l=k+1}^{n}H^{\ast}_{kl}\frac{a_{k}a_{l}}{a_{i}^2}\cos((2\omega_{i}-\omega_{k}+\omega_{l})t) \nonumber \\
    & +2\sum_{k=1}^{n}\sum_{l=k+1}^{n}H^{\ast}_{kl}\frac{a_{k}a_{l}}{a_{i}^2}\cos((2\omega_{i}+\omega_{k}+\omega_{l})t)+2\sum_{k=1}^{n}\sum_{l=k+1}^{n}H^{\ast}_{kl}\frac{a_{k}a_{l}}{a_{i}^2}\cos((2\omega_{i}-\omega_{k}-\omega_{l})t)\,, \label{eq:DeltaCalligraHiiHat}  \\
    \widehat{\!\!\!\!\!\!\!\!\mbox{\calligra H}}^{~~\ast}_{~~ij}(t)&= H^{\ast}_{ij}  {+}\Delta~~~~\! \widehat{\mbox{\calligra \!\!\!\!\!\!\!\! H}}^{~~\ast}_{~~ij}(t) 
    + N_{ij}(t)\left[\frac{1}{2}\tilde{\theta}^{\top}(t)H^{\ast}\tilde{\theta}(t) {+}S^{\top}(t)H^{\ast}\tilde{\theta}(t)\right]\,, \label{eq:calligraHijHat}
\end{align}
\begin{align}
    \Delta~~~~\! \widehat{\mbox{\calligra \!\!\!\!\!\!\!\! H}}^{~~\ast}_{~~ij}(t)&= -H_{ij}^{\ast}\cos(2\omega_{i}t)-H_{ij}^{\ast}\cos(2\omega_{j}t) +\frac{H_{ij}^{\ast}}{2}\cos(2(\omega_{i}-\omega_{j})t)+\frac{H_{ij}^{\ast}}{2}\cos(2(\omega_{i}+\omega_{j})t) \nonumber  \\
    & + \frac{1}{2}\sum_{\substack{k=1 \\ k\neq i}}^{n}H^{\ast}_{kk}\frac{a_{k}^{2}}{a_{i}a_{j}}\left[\cos(2(\omega_{i}-\omega_{j})t)+\cos(2(\omega_{i}+\omega_{j})t)\right] -\frac{1}{2}\sum_{\substack{k=1 \\ k\neq i}}^{n}H^{\ast}_{kk}\frac{a_{k}^{2}}{a_{i}a_{j}}\cos((\omega_{i}-\omega_{j}+2\omega_{k})t) \nonumber  \\
    & -\frac{1}{2}\sum_{\substack{k=1 \\ k\neq i}}^{n}H^{\ast}_{kk}\frac{a_{k}^{2}}{a_{i}a_{j}}\cos((\omega_{i}-\omega_{j}-2\omega_{k})t) +\frac{1}{2}\sum_{\substack{k=1 \\ k\neq i}}^{n}H^{\ast}_{kk}\frac{a_{k}^{2}}{a_{i}a_{j}}\cos((\omega_{i}+\omega_{j}+2\omega_{k})t) \nonumber\\
    &+\frac{1}{2}\sum_{\substack{k=1 \\ k\neq i}}^{n}H^{\ast}_{kk}\frac{a_{k}^{2}}{a_{i}a_{j}}\cos((\omega_{i}+\omega_{j}-2\omega_{k})t)+\frac{1}{2}\sum_{\substack{k=1 \\ k\neq i}}^{n}\sum_{l=k+1}^{n}H^{\ast}_{kl}\frac{a_{k}a_{l}}{a_{i}a_{j}}\cos((\omega_{i}-\omega_{j}+\omega_{k}-\omega_{l})t) \nonumber\\
    &+\frac{1}{2}\sum_{\substack{k=1 \\ k\neq i}}^{n}\sum_{l=k+1}^{n}H^{\ast}_{kl}\frac{a_{k}a_{l}}{a_{i}a_{j}}\cos((\omega_{i}-\omega_{j}-\omega_{k}+\omega_{l})t) -\frac{1}{2}\sum_{\substack{k=i+1 \\ k\neq j}}^{n}H^{\ast}_{ik}\frac{a_{k}}{a_{j}}\cos((\omega_{j}+\omega_{k})t)+N_{ij}(t)Q^{\ast}\nonumber \\
            & -\frac{1}{2}\sum_{\substack{k=1 \\ k\neq i}}^{n}\sum_{l=k+1}^{n}H^{\ast}_{kl}\frac{a_{k}a_{l}}{a_{i}a_{j}}\cos((\omega_{i}-\omega_{j}+\omega_{k}+\omega_{l})t)  -\frac{1}{2}\sum_{\substack{k=1 \\ k\neq i}}^{n}\sum_{l=k+1}^{n}H^{\ast}_{kl}\frac{a_{k}a_{l}}{a_{i}a_{j}}\cos((\omega_{i}-\omega_{j}-\omega_{k}-\omega_{l})t) \nonumber  \\
    & -\frac{1}{2}\sum_{\substack{k=1 \\ k\neq i}}^{n}\sum_{l=k+1}^{n}H^{\ast}_{kl}\frac{a_{k}a_{l}}{a_{i}a_{j}}\cos((\omega_{i}+\omega_{j}+\omega_{k}-\omega_{l})t)  -\frac{1}{2}\sum_{\substack{k=1 \\ k\neq i}}^{n}\sum_{l=k+1}^{n}H^{\ast}_{kl}\frac{a_{k}a_{l}}{a_{i}a_{j}}\cos((\omega_{i}+\omega_{j}-\omega_{k}+\omega_{l})t)\nonumber \\
        & +\frac{1}{2}\sum_{\substack{k=1 \\ k\neq i}}^{n}\sum_{l=k+1}^{n}H^{\ast}_{kl}\frac{a_{k}a_{l}}{a_{i}a_{j}}\cos((\omega_{i}+\omega_{j}+\omega_{k}+\omega_{l})t)  +\frac{1}{2}\sum_{\substack{k=1 \\ k\neq i}}^{n}\sum_{l=k+1}^{n}H^{\ast}_{kl}\frac{a_{k}a_{l}}{a_{i}a_{j}}\cos((\omega_{i}+\omega_{j}-\omega_{k}-\omega_{l})t) \nonumber  \\ 
        & +\frac{1}{2}\sum_{\substack{k=i+1 \\ k\neq j}}^{n}H^{\ast}_{ik}\frac{a_{k}}{a_{j}}\cos((2\omega_{i}-\omega_{j}-\omega_{k})t)  +\frac{1}{2}\sum_{\substack{k=i+1 \\ k\neq j}}^{n}H^{\ast}_{ik}\frac{a_{k}}{a_{j}}\cos((-\omega_{j}+\omega_{k})t) \nonumber\\
        &- \frac{1}{2}\sum_{\substack{k=i+1 \\ k\neq j}}^{n}H^{\ast}_{ik}\frac{a_{k}}{a_{j}}\cos((2\omega_{i}+\omega_{j}+\omega_{k})t)\,.
\label{eq:DeltaCalligraHijHat} 
\end{align}
Thus, by using (\ref{eq:calligraHiiHat})--(\ref{eq:DeltaCalligraHijHat}), we can define the time-varying matrix $\Delta~~~~\! \widehat{\mbox{\calligra \!\!\!\!\!\!\!\! H}}^{~~\ast}(t) \in  {R}^{n \times n}$ such that
\begin{align}
\Delta~~~~\! \widehat{\mbox{\calligra \!\!\!\!\!\!\!\! H}}^{~~~\ast}(t)&:= \begin{bmatrix}
													\Delta~~~~\! \widehat{\mbox{\calligra \!\!\!\!\!\!\!\! H}}^{~~~\ast}_{~~~11}(t) & \Delta~~~~\! \widehat{\mbox{\calligra \!\!\!\!\!\!\!\! H}}^{~~~\ast}_{~~~12}(t) & \ldots & \Delta~~~~\! \widehat{\mbox{\calligra \!\!\!\!\!\!\!\! H}}^{~~~\ast}_{~~~1n}(t) \\
													\Delta~~~~\! \widehat{\mbox{\calligra \!\!\!\!\!\!\!\! H}}^{~~~\ast}_{~~~21}(t) & \Delta~~~~\! \widehat{\mbox{\calligra \!\!\!\!\!\!\!\! H}}^{~~~\ast}_{~~~22}(t) & \ldots & \Delta~~~~\! \widehat{\mbox{\calligra \!\!\!\!\!\!\!\! H}}^{~~~\ast}_{~~~2n}(t) \\
													\vdots                         & \vdots                         &        & \vdots                         \\
													\Delta~~~~\! \widehat{\mbox{\calligra \!\!\!\!\!\!\!\! H}}^{~~~\ast}_{~~~n1}(t) & \Delta~~~~\! \widehat{\mbox{\calligra \!\!\!\!\!\!\!\! H}}^{~~~\ast}_{~~~n2}(t) & \ldots & \Delta~~~~\! \widehat{\mbox{\calligra \!\!\!\!\!\!\!\! H}}^{~~~\ast}_{~~~nn}(t) \\
												 \end{bmatrix} \,, \label{eq:DeltaCalligraHhat} 
\end{align}
and (\ref{eq:calligraHHat_20250527}) can be rewritten as
\begin{align}
   \widehat{\!\!\!\!\!\!\!\!\mbox{\calligra H}}^{~~~\ast}(t) &:= H^{\ast}+\Delta~~~~\! \widehat{\mbox{\calligra \!\!\!\!\!\!\!\! H}}^{~~~\ast}(t) + N(t)\left[\frac{1}{2}\tilde{\theta}^{\top}(t)H^{\ast}\tilde{\theta}(t) {+}S^{\top}(t)H^{\ast}\tilde{\theta}(t)\right]\,, \label{eq:calligraHhat_20250528_v1}
\end{align}
Since the term $\tilde{\theta}^{\top}(t)H^{\ast}\tilde{\theta}(t)$ is quadratic in $\tilde{\theta}(t)$ and, therefore, may be neglected in a local analysis \cite{AK:2003},  the Hessian estimate (\ref{eq:calligraHhat_20250528_v1}) can be rewritten as 
\begin{align}
   ~~~~\widehat{\!\!\!\!\!\!\!\!\mbox{\calligra H}}^{~~~\ast}(t) &:= H^{\ast} {+}\Delta~~~~\! \widehat{\mbox{\calligra \!\!\!\!\!\!\!\! H}}^{~~~\ast}(t)  {+}S^{\top}(t)H^{\ast}\tilde{\theta}(t)N(t) . \label{eq:calligraHhat_20250528_v2}
\end{align}
The second component is a \emph{Riccati-based adaptation law}, 
\begin{align}
    \frac{d\Gamma}{dt}(t):=\omega_{r}\Gamma(t)-\omega_{r}\Gamma(t)~~~~\widehat{\!\!\!\!\!\!\!\!\mbox{\calligra H}}^{~~~\ast}\!\!(t)\Gamma(t)\,, \label{eq:dotGamma}
\end{align}
designed to compute an estimate of the \emph{inverse of the Hessian}. A key feature of this Riccati-based approach is its robustness: it ensures that a usable inverse estimate is available even when the Hessian approximation is close to singular or poorly conditioned.
\subsection{Newton-based Unit Vector Control Law}
For the Newton-based algorithm depicted in Fig.~2, the corresponding tuning law is given by
\begin{align}
u(t)=-K\frac{\Gamma(t)\hat{G}(t)}{\|\Gamma(t)\hat{G}(t)\|} \,, \quad \forall t\geq 0\,. \label{eq:uNewton}
\end{align}
Thus, to analyze the closed-loop system, we use the estimation error (\ref{eq:tildeThetai_v1}), the gradient estimate (\ref{eq:hatG_20240302_3}) and define Hessian estimate error 
\begin{align}
    \tilde{\Gamma}(t):=\Gamma(t)-H^{\ast-1}\,. \label{eq:tildeGamma}
\end{align}

\subsection{Closed-loop System}
The closed-loop system, in error variables, is given by 
\begin{align}
\frac{d\hat{G}}{dt}(t)& = - \left(H^{\ast}KH^{\ast-1} + H^{\ast}K\tilde{\Gamma}(t){+}\Delta \! \mbox{\calligra H}^{~~~\ast}\!(t)KH^{\ast-1} +\Delta \mbox{\calligra H}^{~~~\ast}(t)K\tilde{\Gamma}(t) \right) \frac{\hat{G}(t)}{\|(H^{\ast-1}+ \tilde{\Gamma}(t))\hat{G}(t)\|} \nonumber\\
& \quad +\frac{d\Delta \! \mbox{\calligra H}^{~~~\ast}}{dt}\!(t)\tilde{\theta}(t)+\frac{d \Delta}{dt}(t)\,, \label{eq:dotHatG_20250603_v1} \\ 
\frac{d\tilde{\theta}}{dt}(t)&=\frac{-K\left[\left(I_{n}+H^{\ast-1}\Delta \! \mbox{\calligra H}^{~~~\ast}\!(t)+\tilde{\Gamma}(t)H^{\ast} \right.\right. \left.\left.+~\tilde{\Gamma}(t)\Delta \! \mbox{\calligra H}^{~~~\ast}\!(t)\right)\tilde{\theta}(t)+(H^{\ast-1}+~\tilde{\Gamma}(t))\Delta(t)\right]}
{\left\|\left(I_{n}+H^{\ast-1}\Delta \! \mbox{\calligra H}^{~~~\ast}\!(t)+\tilde{\Gamma}(t)H^{\ast} .+\tilde{\Gamma}(t)\Delta \! \mbox{\calligra H}^{~~~\ast}\!(t)\right)\tilde{\theta}(t)+(H^{\ast-1}+\tilde{\Gamma}(t))\Delta(t)\right\|}\label{eq:dtildeThetadt_20250529_1}\,,\\
    \frac{d\tilde{\Gamma}}{dt}(t)&=\! {-}\omega_{r}\tilde{\Gamma}(t) -\omega_{r}\tilde{\Gamma}(t)H^{\ast}\tilde{\Gamma}(t) -\omega_{r}\tilde{\Gamma}(t)(\Delta~~~~\! \widehat{\mbox{\calligra  H}}^{~~~\ast}(t) {+}S^{\top}(t)H^{\ast}\tilde{\theta}(t)N(t))H^{\ast {-}1} \nonumber \\
    &\quad {-}\omega_{r}H^{\ast {-}1}(\Delta~~~~\! \widehat{\mbox{\calligra \!\!\!\!\!\!\!\! H}}^{~~~\ast}(t) {+}S^{\top}(t)H^{\ast}\tilde{\theta}(t)N(t))\tilde{\Gamma}(t) - \omega_{r}\tilde{\Gamma}(t)(\Delta~~~~\! \widehat{\mbox{\calligra \!\!\!\!\!\!\!\! H}}^{~~~\ast}(t) {+}S^{\top}(t)H^{\ast}\tilde{\theta}(t)N(t))\tilde{\Gamma}(t) \nonumber \\
     &-\omega_{r}H^{\ast {-}1}(\Delta~~~~\! \widehat{\mbox{\calligra \!\!\!\!\!\!\!\! H}}^{~~~\ast}(t) {+}S^{\top}(t)H^{\ast}\tilde{\theta}(t)N(t))H^{\ast {-}1}\,.\label{eq:dotTildeGamma_20250529_1}
\end{align}

\subsection{Time Scaling}

By using the transformation $\bar{t}=\omega t$, with (\ref{eq:omega_scaled}), it is possible to rewrite the dynamics (\ref{eq:dotHatG_20250603_v1})--(\ref{eq:dotTildeGamma_20250529_1}) in a different time-scale such that
\begin{align}
\frac{d\hat{G}(\bar{t})}{d\bar{t}}&= {-}\frac{1}{\omega}\left(H^{\ast}KH^{\ast-1} {+}H^{\ast}K\tilde{\Gamma}(\bar{t}) {+}\Delta \! \mbox{\calligra H}^{~~~\ast}\!(\bar{t})KH^{\ast-1} +\Delta \! \mbox{\calligra H}^{~~~\ast}\!(\bar{t})K\tilde{\Gamma}(\bar{t}) \right)\frac{\hat{G}(\bar{t})}{\|(H^{\ast-1}+\tilde{\Gamma}(\bar{t}))\hat{G}(\bar{t})\|} \nonumber \\
&\quad {+}\frac{d\Delta \! \mbox{\calligra H}^{~~~\ast}}{d\bar{t}}\!(t)\tilde{\theta}(\bar{t}) {+}\frac{d \Delta}{d\bar{t}}(\bar{t})\,, \label{eq:dotHatGav_event_4_siso} \\
\frac{d\tilde{\theta}}{dt}(t)&=\frac{-K\left[\left(I_{n}+H^{\ast-1}\Delta \! \mbox{\calligra H}^{~~~\ast}\!(t)+\tilde{\Gamma}(t)H^{\ast} \right.\right. \left.\left.+~\tilde{\Gamma}(t)\Delta \! \mbox{\calligra H}^{~~~\ast}\!(t)\right)\tilde{\theta}(t)+(H^{\ast-1}+~\tilde{\Gamma}(t))\Delta(t)\right]}{\left\|\left(I_{n}+H^{\ast-1}\Delta \! \mbox{\calligra H}^{~~~\ast}\!(t)+\tilde{\Gamma}(t)H^{\ast} .+\tilde{\Gamma}(t)\Delta \! \mbox{\calligra H}^{~~~\ast}\!(t)\right)\tilde{\theta}(t)+(H^{\ast-1}+\tilde{\Gamma}(t))\Delta(t)\right\|}\label{eq:dtildeThetadt_20250529_2}\,,\\
    \frac{d\tilde{\Gamma}}{d\bar{t}}(\bar{t})&=\! {-}\frac{\omega_{r}}{\omega}\tilde{\Gamma}(\bar{t}) {-}\frac{\omega_{r}}{\omega}\tilde{\Gamma}(\bar{t})H^{\ast}\tilde{\Gamma}(\bar{t}) -\frac{\omega_{r}}{\omega}\tilde{\Gamma}(\bar{t})(\Delta~~~~\! \widehat{\mbox{\calligra \!\!\!\!\!\!\!\! H}}^{~~~\ast}(\bar{t}) {+}S^{\top}(\bar{t})H^{\ast}\tilde{\theta}(\bar{t})N(\bar{t}))H^{\ast {-}1} \nonumber \\
    &-\frac{\omega_{r}}{\omega}H^{\ast {-}1}(\Delta~~~~\! \widehat{\mbox{\calligra \!\!\!\!\!\!\!\! H}}^{~~~\ast}(\bar{t}) {+}S^{\top}(\bar{t})H^{\ast}\tilde{\theta}(\bar{t})N(\bar{t}))\tilde{\Gamma}(\bar{t}) -\frac{\omega_{r}}{\omega}\tilde{\Gamma}(\bar{t})(\Delta~~~~\! \widehat{\mbox{\calligra \!\!\!\!\!\!\!\! H}}^{~~~\ast}(\bar{t}) {+}S^{\top}(\bar{t})H^{\ast}\tilde{\theta}(\bar{t})N(\bar{t}))\tilde{\Gamma}(\bar{t}) \nonumber \\
     &\quad {-}\frac{\omega_{r}}{\omega}H^{\ast {-}1}(\Delta~~~~\! \widehat{\mbox{\calligra \!\!\!\!\!\!\!\! H}}^{~~~\ast}(\bar{t}) {+}S^{\top}(\bar{t})H^{\ast}\tilde{\theta}(\bar{t})N(\bar{t}))H^{\ast {-}1}\,, \label{eq:dotGamma_1_event_siso}
\end{align}
Now, by defining the augmented state
\begin{align}
X^{\top}(\bar{t}):=\begin{bmatrix}\hat{G}^{\top}(\bar{t})\,, \tilde{\theta}^{\top}(\bar{t})\,, \tilde{\Gamma}^{\top}(\bar{t}) 
\end{bmatrix}^{\top}\,,
\end{align}
one arrives at the dynamics
\begin{align}
\dfrac{dX(\bar{t})}{d\bar{t}}&=\dfrac{1}{\omega}\mathcal{F}\left(\bar{t},X,\dfrac{1}{\omega}\right)\,, \quad \label{eq:dotX_event_siso}
\mathcal{F}^{\top}=\begin{bmatrix} \mathcal{F}_{1}\,, \mathcal{F}_{2}\,, \mathcal{F}_{3} \end{bmatrix}\,.
\end{align}

\subsection{Average Closed-loop Sysem}

Due to the discontinuous nature of the proposed control strategy, the averaging theory for discontinuous systems is employed in the paper, according to \cite{P:1979}.

The augmented system (\ref{eq:dotX_event_siso}) has a small parameter $1/\omega$ as well as a $T$-periodic function $\mathcal{F}\left(\bar{t},X,\dfrac{1}{\omega}\right)$ in $\bar{t}$, hence it can be studied through the averaging method for stability analysis by averaging  $\mathcal{F}\left(\bar{t},X,\dfrac{1}{\omega}\right)$ at $\displaystyle \lim_{\omega\to \infty}\dfrac{1}{\omega}=0$, as shown in \cite{P:1979}, {\it i.e.},
\begin{align}
\dfrac{dX_{\text{av}}(\bar{t})}{d\bar{t}}&=\dfrac{1}{\omega}\mathcal{F}_{\text{av}}\left(X_{\text{av}}\right) \,, \label{eq:dotXav_event_1_siso} \\
\mathcal{F}_{\text{av}}\left(X_{\text{av}}\right)&=\dfrac{1}{T}\int_{0}^{T}\mathcal{F}\left(\delta,X_{\text{av}},0\right)d\delta
\,.  \label{eq:mathcalFav_event_siso}
\end{align}
Basically, the problem in the averaging method is to determine in what sense the behavior of the autonomous system (\ref{eq:dotXav_event_1_siso}) approximates the behavior of the nonautonomous system (\ref{eq:dotX_event_siso}) such that (\ref{eq:dotX_event_siso}) can be represented as a perturbed version of the system (\ref{eq:dotXav_event_1_siso}).

Therefore, treating the non-periodic states $\hat{G}(\bar{t})$, $\tilde{\theta}(\bar{t})$ and $\Gamma(\bar{t})$ as constants in (\ref{eq:dotHatGav_event_4_siso})--(\ref{eq:dotX_event_siso}), one gets the following average closed-loop system   
\begin{align}
    \frac{d\hat{G}_{\rm av}}{d\bar{t}}(\bar{t})=&-\frac{1}{\omega}H^{\ast}KH^{\ast -1}\frac{\hat{G}_{\rm av}(\bar{t})}{\|(H^{\ast -1}+\tilde{\Gamma}_{\rm av}(\bar{t}))\hat{G}_{\rm av}(\bar{t})\|} \label{eq:dhatGdt_newton2}\,, \\
\frac{d\tilde{\theta}_{\rm av}}{d\bar{t}}(\bar{t}) = & -\frac{1}{\omega}K\frac{(I_{n}+\tilde{\Gamma}(\bar{t})H^{\ast}) \tilde{\theta}_{\rm av}(\bar{t})}{\|(I_{n}+\tilde{\Gamma}(\bar{t})H^{\ast}) \tilde{\theta}_{\rm av}(\bar{t})\|}\label{eq:dtildeThetadt_newton2}\,,\\
\frac{d\tilde{\Gamma}_{\rm av}}{d\bar{t}}(\bar{t}) = & ~-\frac{\omega_{r}}{\omega}\tilde{\Gamma}_{\rm av}(\bar{t}) - \frac{\omega_{r}}{\omega}\tilde{\Gamma}_{\rm av}(\bar{t})H^{\ast}\tilde{\Gamma}_{\rm av}(\bar{t})  \label{eq:dtilde_gamma_newton2}\,.
\end{align}

The system (\ref{eq:dhatGdt_newton2})--(\ref{eq:dtilde_gamma_newton2}) is nonlinear due the quadratic terms $\tilde{\Gamma}_{\rm{av}}(\bar{t})\hat{G}_{\rm{av}}(\bar{t})$, $\tilde{\Gamma}_{\rm av}(\bar{t})H^{\ast}\tilde{\theta}_{\rm av}(\bar{t})$ and $\tilde{\Gamma}_{\rm av}(\bar{t})H^{\ast}\tilde{\Gamma}_{\rm av}(\bar{t})$. To analyze its stability properties, we provide a linearization to approximate of the system's behavior in the vicinity of origin. In this sense, for small values of $\hat{G}_{\rm{av}}(0)$, $\tilde{\theta}_{\rm{av}}(0)$ and $\tilde{\Gamma}_{\rm{av}}\left(0\right)$, the quadratic terms are negligible in the linearized corresponding system such that the system (\ref{eq:dhatGdt_newton2})--(\ref{eq:dtilde_gamma_newton2}) can be rewritten as
\begin{align}
    \frac{d\hat{G}_{\rm av}}{d\bar{t}}(\bar{t})=&-\frac{1}{\omega}H^{\ast}KH^{\ast -1}\frac{\hat{G}_{\rm av}(\bar{t})}{\|H^{\ast -1}\hat{G}_{\rm av}(\bar{t})\|}  \label{eq:dhatGdt_newton3} \,, \\
    \frac{d\tilde{\theta}_{\rm av}}{d\bar{t}}(\bar{t}) = & -\frac{1}{\omega}K\frac{ \tilde{\theta}_{\rm av}(\bar{t})}{\|\tilde{\theta}_{\rm av}(\bar{t})\|} \label{eq:dtilde_theta_newton3}\,,\\
    \frac{d\tilde{\Gamma}_{\rm av}}{d\bar{t}}(\bar{t}) = & ~-\frac{\omega_{r}}{\omega}\tilde{\Gamma}_{\rm av}(\bar{t})\label{eq:dtilde_gamma_newton3}\,.
\end{align}

\subsection{Stability Analysis}
The next theorem guarantees the finite-time stability of the proposed Newton-ESC with unit vector, as depicted in Fig.~\ref{fig:BD_newton_ESC_UVC}. 

\begin{theo}
Consider the closed-loop average dynamics of the gradient estimate (\ref{eq:dhatGdt_newton3}) and that Assumption 1 holds. For a sufficiently large $\omega > 0$, defined in (\ref{eq:omega_scaled}), there exist a finite-time $t_{S} \in \left(0,\frac{2\lambda_{max}(P)}{\lambda_{min}(Q)}\|\tilde{\theta}_{\rm av}(0)\|\right]$, such
that the sliding motion on $\hat{G}_{\rm av}(t) \equiv 0$ occurs. Furthermore, for the non-average system (\ref{eq:gradient_dhatGdt_20250206_1}), the input converge to an error of order $\mathcal{O}(a + \frac{1}{\omega})$ 
\begin{align}
  \| \theta(t) - \theta^* \|& \leq  \left\{
  \begin{array}{ll}
   \sqrt{\frac{\lambda_{\max}(P)}{\lambda_{\min}(P)}} \| \theta(0)  {-} \theta^* \| - \dfrac{1}{2} \frac{\lambda_{\min}(Q)}{\sqrt{\lambda_{\max}(P)}\sqrt{\lambda_{\min}(P)}} t + \mathcal{O} \left( a + \frac{1}{\omega} \right), & \forall t \in [0, t_{S}), \\
  \mathcal{O} \left( a + \frac{1}{\omega} \right), & \forall t \in [t_{S}, +\infty).
  \end{array}
  \right. \label{eq:theta_minimal_v2}
\end{align}
and for the output,
\begin{align}
 |y(t) - Q^{\ast}| &\leq  \left\{
  \begin{array}{ll}
2\|H^{\ast}\|\left[\left(\sqrt{\frac{\lambda_{\max}(P)}{\lambda_{\min}(P)}} \| \theta(0) - \theta^* \| 
        - \frac{1}{2} \frac{\lambda_{\min}(Q)}{
        \sqrt{\lambda_{\max}(P)} \sqrt{\lambda_{\min}(P)}} t \right)^{2} + \mathcal{O}\left(a^{2} + \frac{1}{\omega^{2}}\right)\right]\,, &\forall t \in [0, t_{S})\,, \\
     \mathcal{O}\left(a^{2} + \frac{1}{\omega^{2}}\right) \,, &\forall \in [t_{S}, +\infty)\\
  \end{array}
  \right. \label{eq:output_minimal_v2}
\end{align}
\end{theo}
\begin{proof}
From a Lyapunov function candidate,
\begin{align}
    V_{\rm av}(\bar{t}) = \tilde{\theta}^{\top}_{\rm av}(\bar{t})P\tilde{\theta}_{\rm av}, \quad P = P^{\top} > 0 \label{eq:lyapnuv_tildetheta}\,,
\end{align}
the gain \textcolor{black}{$-K$ is Hurwitz}, by a given $Q = Q^{\top} > 0$, exist $ P = P^{\top}$, a symmetric matrix, such that the Lyapunov's function is $(-K)^{\top}P + P(-K) = -Q$. Thus, through a time scaling and time derivative of (\ref{eq:lyapnuv_tildetheta}) is 
\begin{align}
    \frac{dV_{\rm av}(\bar{t})}{d\bar{t}} & = ~ \dot{\tilde{\theta}}^{\top}_{\rm av}(\bar{t})P\tilde{\theta}_{\rm av} + \tilde{\theta}^{\top}_{\rm av}(\bar{t})P\dot{\tilde{\theta}}_{\rm av}\,, \nonumber\\
    & = ~ \frac{1}{\omega}\frac{\tilde{\theta}^{\top}_{\rm av}(\bar{t})(-K)^{\top}P\tilde{\theta}_{\rm av}(\bar{t})}{\|\tilde{\theta}_{\rm av}(\bar{t})\|} + \frac{1}{\omega}\frac{\tilde{\theta}^{\top}_{\rm av}(\bar{t})P(-K)\tilde{\theta}_{\rm av}(\bar{t})}{\|\tilde{\theta}_{\rm av}(\bar{t})\|}\,,\\
    & = ~ \frac{1}{\omega}\frac{\tilde{\theta}^{\top}_{\rm av}(\bar{t})(-K)^{\top}P + P(-K)\tilde{\theta}_{\rm av}(\bar{t})}{\|\tilde{\theta}_{\rm av}(\bar{t})\|}\,,\\
    & = ~ -\frac{1}{\omega}\frac{\tilde{\theta}^{\top}_{\rm av}(\bar{t})Q\tilde{\theta}_{\rm av}(\bar{t})}{\|\tilde{\theta}_{\rm av}(\bar{t})\|} \label{eq:lyapunov_tildetheta_q} \,.
\end{align}
Now applying the Rayleigh-Ritz inequality \cite{K:2002}
\begin{align}
    \lambda_{min}(Q)||\tilde{\theta}_{\rm av}(t)\|^{2} \leq V_{\rm av}(t) \leq \lambda_{max}(Q)||\tilde{\theta}_{\rm av}(t)\|^{2}, \label{eq:tildetheta_Q_newton}
\end{align}
the Lyapunov candidate derivate (\ref{eq:lyapunov_tildetheta_q}) is
\begin{align}
    \frac{dV_{\rm av}(\bar{t})}{d\bar{t}} \leq -\frac{1}{\omega}\lambda_{min}(Q)\|\tilde{\theta}_{\rm av}(\bar{t})\|\,. \label{eq:dvav_lambdaQ}
\end{align}
Now, by using the Rayleigh-Ritz inequality \cite{K:2002} to the P matrix, 
\begin{align}
    \lambda_{\min}(P)\|\tilde{\theta}_{\rm av}(\bar{t})\|^{2}\leq V_{\rm av}(\bar{t}) \leq \lambda_{\max}(P)\|\tilde{\theta}_{\rm av}(\bar{t})\|^{2}\,,\label{eq:rayleigh-ritz_newton_P}
\end{align}
equation (\ref{eq:dvav_lambdaQ}) is
\begin{align}
    \frac{dV_{\rm av}(\bar{t})}{d\bar{t}} \leq -\frac{1}{\omega}\frac{\lambda_{\min}(Q)}{\sqrt{\lambda_{\max}(P)}}\sqrt{V_{\rm av}(\bar{t})}\,.
    \label{eq:lyapunov_derivative_3}
\end{align}
To solve this inequality, the Comparison Lemma \cite{K:2002} is used,
\begin{align}
    V_{\rm av}(t) &\leq \bar{V}_{\rm av}(\bar{t})\,,~\forall \bar{t} \geq 0\,,\label{eq:comparison_lemma_2}
\end{align}
results, 
\begin{align}
    \frac{d\bar{V}_{\rm av}(\bar{t})}{d\bar{t}} & = -\frac{1}{\omega}\frac{\lambda_{\min}(Q)}{\sqrt{\lambda_{\max}(P)}}\sqrt{\bar{V}_{\rm av}(\bar{t})}\,, \quad \bar{V}_{\rm av}(0)=V_{\rm av}(0)\,. \label{eq:barV_v1}
\end{align}
The solution of (\ref{eq:barV_v1}) with the initial conditions $\bar{V}_{\rm av}(0)$ is given by
\begin{align}
    \sqrt{\bar{V}_{\rm av}(\bar{t})} = \sqrt{\bar{V}_{\rm av}(0)}  -\frac{1}{2\omega}\frac{\lambda_{\min}(Q)}{\sqrt{\lambda_{\max}(P)}}\bar{t}\,.
    \label{eq:Vav_newton_solution}
\end{align}
By using (\ref{eq:rayleigh-ritz_newton_P}), the equation (\ref{eq:Vav_newton_solution}) can be upper bounded by,
\begin{align}
    \sqrt{\lambda_{\min}(P)}\|\tilde{\theta}_{\rm av}(\bar{t})\| \leq &
    \sqrt{\lambda_{\max}(P)}\|\tilde{\theta}_{\rm av}(0)\| - \frac{1}{2\omega}\frac{\lambda_{\min}(Q)}{\sqrt{\lambda_{\max}(P)}}\bar{t}\,,
    \end{align}
thus,
    \begin{align}
    \|\tilde{\theta}_{\rm av}(\bar{t})\| \leq & \sqrt{\frac{\lambda_{max}(P)}{\lambda_{min}(P)}}\|\tilde{\theta}_{\rm av}(0)\| - \frac{1}{2\omega}\frac{\lambda_{\min}(Q)}{\sqrt{\lambda_{\max}(P)}\sqrt{\lambda_{\min}(P)}}\bar{t}\,. \label{eq:thetaav_newton_solution}
\end{align}
From (\ref{eq:thetaav_newton_solution}), there exist a finite-time $\bar{t}_{S} \in  \left(0,\frac{2\omega\lambda_{max}(P)}{\lambda_{min}(Q)}\|\tilde{\theta}_{\rm av}(0)\|\right]$, such that the sliding motion occurs, 
\begin{align}
    t_{S} \in \left(0,~2\frac{\lambda_{max}(P)}{\lambda_{min}(Q)} \|\theta(0) - \theta^{*}\|\right]\,. 
    \label{eq:time_ts_newton}
\end{align}
Since (\ref{eq:gradient_dhatGdt_20250206_1}) has a discontinuous right-hand side, but is also $T$-periodic in $t$, and noting that the average system with state 
$\tilde{\theta}_{\rm av}(\bar{t})$ is finite-time stable according to (\ref{eq:thetaav_newton_solution}) and (\ref{eq:time_ts_newton}), we can invoke the averaging theorem in~\cite{P:1979} to conclude that
\begin{align}
    \|\tilde{\theta}(t) - \tilde{\theta}_{\rm av}(t)\| \leq \mathcal{O}\left(\frac{1}{\omega}\right). \label{eq:tilde_theta_theta_av2}
\end{align}
By applying the triangle inequality \cite{A:1957}, we also obtain:
\begin{align}
\|\tilde{\theta}(t)\| 
&\leq \|\tilde{\theta}_{\rm av}(t)\| + \mathcal{O}\left(\frac{1}{\omega} \right)
\leq 
\begin{cases}
\sqrt{\dfrac{\lambda_{\max}(P)}{\lambda_{\min}(P)}} \|\tilde{\theta}_{\rm av}(0)\| 
- \dfrac{1}{2} \dfrac{ \lambda_{\min}(Q)}{\sqrt{\lambda_{\max}(P)} \sqrt{\lambda_{\min}(P)}} t  + \mathcal{O}\left(\frac{1}{\omega} \right) 
& \forall t \in [0, t_{S}), \\
  \mathcal{O}\left(\frac{1}{\omega} \right)\,, &\forall t \in [t_{S}, \infty). \label{eq:normTildeTheta}
\end{cases}
\end{align}
Now, backing to the non-average system, from (\ref{eq:S_v1}) and from Fig. ~\ref{fig:BD_ESC_UVC}, we can verify that,
\begin{align}
    \theta(t) - \theta^{\ast} = \tilde{\theta}(t) + S(t)\,,
    \label{eq:theta_fig2}
\end{align}
whose Euclidian norm satisfies,
\begin{align}
    \| \theta(t) - \theta^* \| &= \| \tilde{\theta}(t) + S(t) \| \leq \| \tilde{\theta}(t) \| + \| S(t) \| \nonumber \\
    &\leq
    \begin{cases} 
    \sqrt{\dfrac{\lambda_{\max}(P)}{\lambda_{\min}(P)}} \|\theta(0) - \theta^{\ast}\|
- \dfrac{1}{2} \dfrac{ \lambda_{\min}(Q)}{\sqrt{\lambda_{\max}(P)} \sqrt{\lambda_{\min}(P)}} \, t + ~\mathcal{O}\left(a+ \dfrac{1}{\omega}\right), 
& \forall t \in [0, t_{S}),\\
         ~\mathcal{O}\left(a+ \dfrac{1}{\omega}\right), ~~ &\forall t \in [t_{S}, +\infty).
    \end{cases}
\label{eq:gradient_theta_diff_20250701}
\end{align}
From (\ref{eq:gradient_theta_diff_20250701}), the control strategy (\ref{eq:uNewton}), influenced by the Hessian, $H^{\ast}$, allows the convergence to the residual set $\mathcal{O}\left(a + \frac{1}{\omega}\right)$ in a finite-time, within the interval  
\begin{align}
    t_{S} \in \left(0,~2\frac{\lambda_{max}(P)}{\lambda_{min}(Q)} \|\theta(0) - \theta^{*}\|\right] \,.
\end{align}
From (\ref{eq:y_v2}) and the Cauchy-Schwarz inequality \cite{S:2008}, it is possible to write 
\begin{align}
    |y(t) - Q^{*}|=|(\theta(t) - \theta^{\ast})^{T}H^{\ast}(\theta(t) - \theta^{\ast})|\leq \|H^{\ast}\|\|\theta(t) - \theta^{*}\|^{2}\,,
    \label{eq:gradient_norm_y_20250701}
\end{align}
and substituting (\ref{eq:gradient_theta_diff_20250701}) in (\ref{eq:gradient_norm_y_20250701}), lead to the following upper bound
\begin{align}
|y(t) - Q^{*}| & 
\leq \left\{
\begin{array}{ll}
\|H^{\ast}\|\left[
     \sqrt{\frac{\lambda_{\max}(P)}{\lambda_{\min}(P)}} \| \theta(0) - \theta^* \| - \frac{1}{2} \frac{ \lambda_{\min}(Q)}{
    \sqrt{\lambda_{\max}(P)} \sqrt{\lambda_{\min}(P)}} t 
    + \mathcal{O}\left(a + \frac{1}{\omega}\right)
\right]^2\,, &\forall t \in (0, t_{S}], \\
      \mathcal{O}\left(a+ \frac{1}{\omega}\right)^{2}, & \forall t \in (t_{S}, +\infty) .\label{eq:gradient_y_norm_v32_20250701} \\
  \end{array}\right. \\
& \leq \left\{
  \begin{array}{ll}
2\|H^{\ast}\|\left(\sqrt{\frac{\lambda_{\max}(P)}{\lambda_{\min}(P)}} \| \theta(0) - \theta^* \| 
        - \frac{1}{2} \frac{\lambda_{\min}(Q)}{
        \sqrt{\lambda_{\max}(P)} \sqrt{\lambda_{\min}(P)}} t \right)^{2} + \mathcal{O}\left(a^{2} + \frac{1}{\omega^{2}}\right),\hspace{-0.65cm} &\forall t \in (0, t_{S}],\\
      \mathcal{O}\left(a^{2} + \frac{1}{\omega^{2}}\right) ,\hspace{-0.65cm} & \forall t \in (t_{S}, +\infty) .\label{eq:gradient_y_norm_v3_20250701}
  \end{array}\right.
\end{align}
From (\ref{eq:gradient_y_norm_v3}), the convergence of $|y(t) - Q^*|$ to the residual set $\mathcal{O}\left(a^2 + \frac{1}{\omega^{2}}\right)$ occurs for $t > t_{S}$. Moreover, the solution of (\ref{eq:dtilde_gamma_newton3}) is
\begin{align}
    \|\tilde{\Gamma}_{\rm av}(\bar{t})\| = \exp(-\frac{\omega_{r}}{\omega}\bar{t})\|\tilde{\Gamma}_{\rm av}(0)\|\,, \quad \forall \bar{t} \geq 0\,.\label{eq:norm_tildetheta}
\end{align}
Since (\ref{eq:norm_tildetheta}) is $T$-periodic in $t$, $1/\omega$ is a positive small parameter, the origin $\tilde{\Gamma}_{\rm{av}}=0$ is at least an exponentially stable equilibrium point of the closed-loop. Then, by invoking  \cite{P:1979}, there exists an upper bound for (\ref{eq:norm_tildetheta}) such that
\begin{align}
\|\Gamma(t)-H^{\ast-1}\|&\leq\exp\left(-\omega_{r} t\right)\|\Gamma(0)-H^{\ast-1}\|+\mathcal{O}\left(a+\frac{1}{\omega}\right)\,. \label{eq:norm_tildegamma_mimo}
\end{align}
Then, the proof is completed. \hfill$\square$
\end{proof}
\section{Simulation Results}

To illustrate the results between the unit vector gradient-based and unit vector Newton-based extremum seeking methods, the following parameters were used \cite{GKN:2012}: $\delta = 0.1$, $\omega$ = 0.1 rad/s, $\omega_{1}$ = $70\omega$, $\omega_{2}$ = $50\omega$, $\omega'_{L}$ = 10, $\omega'_{H}$ = 8, $\omega'_{R}$ = 10, $a = [0.1 ~0.1]^{\top}$, $K^{''}_{g}$ = $10^{-4}$diag([-250 ~-250]), $K^{''}_{n}$ = diag([1~1]), $\Gamma^{-1}_{0}$ = 400diag([1~1]), $\tilde{\theta}_{0}$ = $[2.5~5]^{\top}$, $Q^{\ast}$ = 100, $\theta^{\ast} = [2~4]^{\top}$, \textcolor{black}{$H^{\ast}_{11} = 100, H^{\ast}_{12} = H^{\ast}_{21}= 30$ and $H^{\ast}_{22} = 20$.} Fig. \ref{fig:uvc} and Fig. \ref{fig:uvc_newton}~ show the simulation of the gradient-based method and Newton-based method to two variables over 0 to 1000 seconds, respectively. The dotted line represents the convergence value. Fig. \ref{fig:uvc}(a) illustrate the control action, having low convergence rate and around to 600 seconds the sliding takes place. In Fig. \ref{fig:uvc}(b) depicts the evolution of time to estimate the maximum output $Q^{\ast} = 100$. The gradient estimate $\hat{G}_{1}(t)$ (blue line) converges faster than $\hat{G}_{2}(t)$ (red line) around 400 seconds and 600 seconds, respectively, as illustrated in Fig. \ref{fig:uvc}(c). So as the variable $\theta(t)$, converging to $\theta_1^*=4$ (red line) and $\theta_2^*=2$ (blue line) simultaneously approximately 600 seconds. Now, for the Newton-based method, the convergence rate is faster than the gradient-based. As illustrated in Fig. \ref{fig:uvc_newton}(a), the sliding takes place around 170 seconds, while the maximum output was around 150 seconds, as shown in Fig. \ref{fig:uvc_newton}(b). Fig.~\ref{fig:uvc_newton}(e) illustrates the inverse Hessian elements and their convergence to the values $H^{\ast -1}_{11}$ = 0.0182, $H^{\ast -1}_{22}$ = 0.0909 and $H^{\ast -1}_{12}$ = $H^{\ast -1}_{22}$ = -0.0273 around 100 seconds, before $\hat{G}(t)$ and $\theta(t)$, as depicted on Fig. \ref{fig:uvc_newton}(c) and Fig. \ref{fig:uvc_newton}(d).
\begin{figure*}[!ht]
	\centering
	\subfigure[Control law $u(t)$. \label{fig:u_uvc}]{\includegraphics[width=5.5cm]{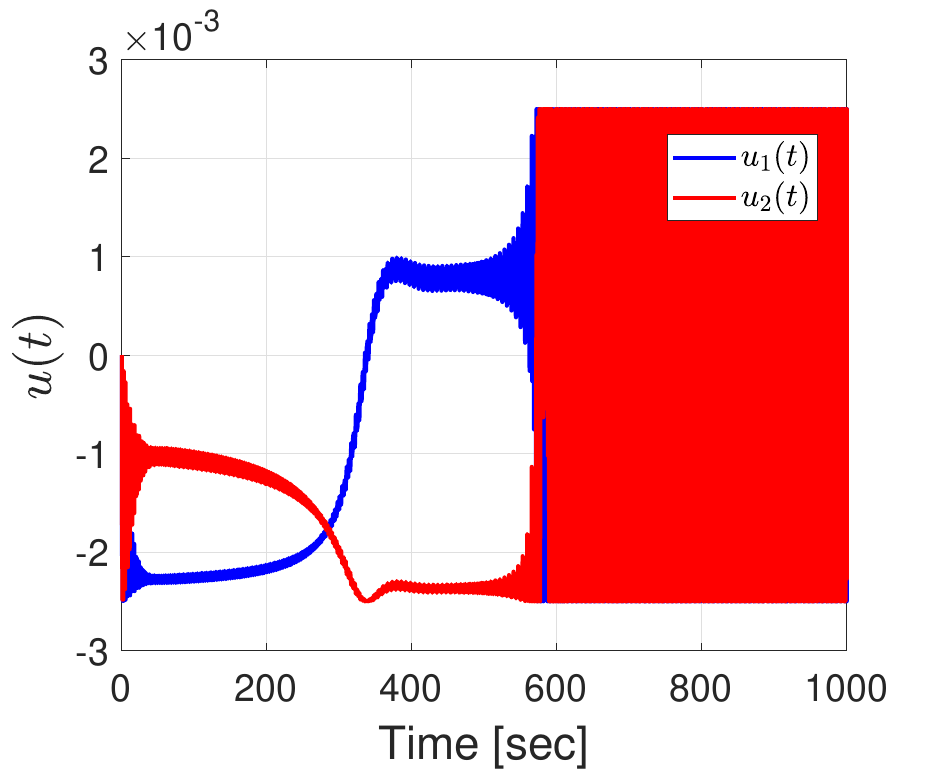}}
	\subfigure[Output $y(t)$. \label{fig:y_uvc}]{\includegraphics[width=5.5cm]{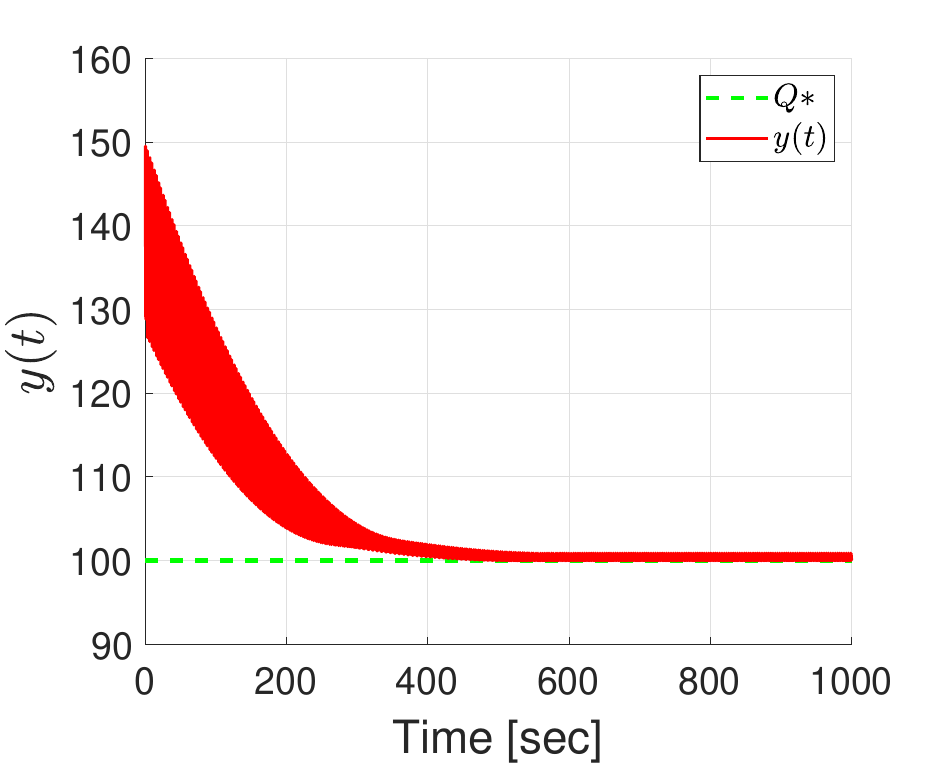}}
		\\
		\subfigure[Gradient estimate $\hat{G}(t)$. \label{fig:ghat_uvc}]{\includegraphics[width=5.5cm]{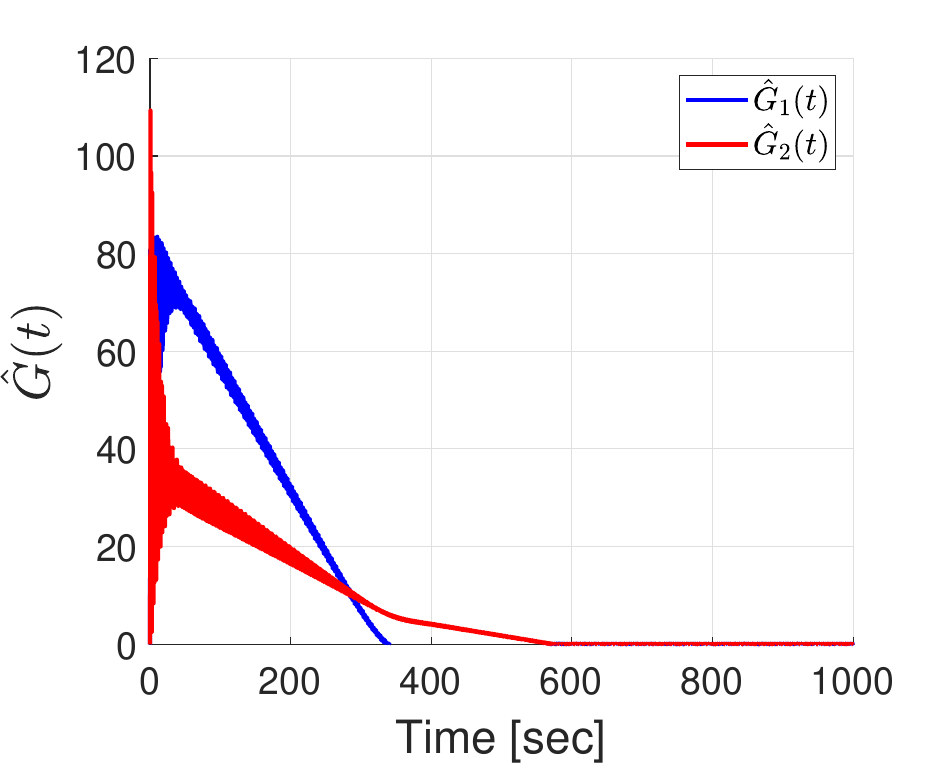}}
		\subfigure[Input $\theta(t)$. \label{fig:theta_uvc}]{\includegraphics[width=5.5cm]{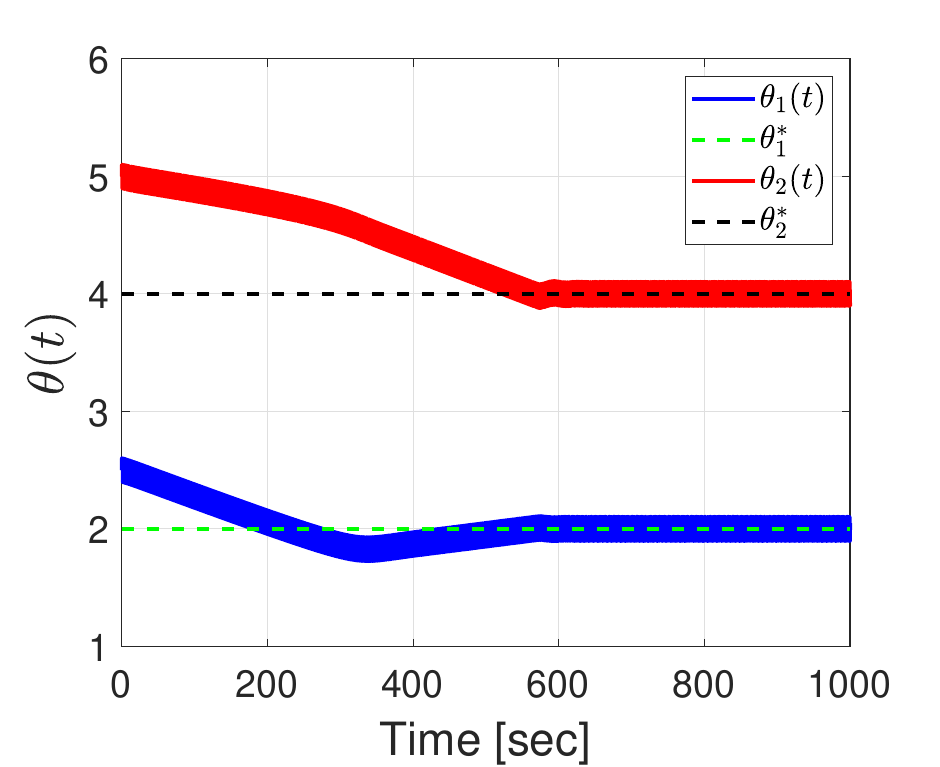}}
    \caption{The simulation variables are: (a) control signal $u(t)$, (b) output dsignal $y(t)$, (c) gradient estimate $\hat{G}(t)$, and (d) input signal $\theta(t)$.}
    \label{fig:uvc}
\end{figure*}
\begin{figure*}[!ht]
	\centering
	\subfigure[Control law $u(t)$. \label{fig:u_uvc_newton}]{\includegraphics[width=5.5cm]{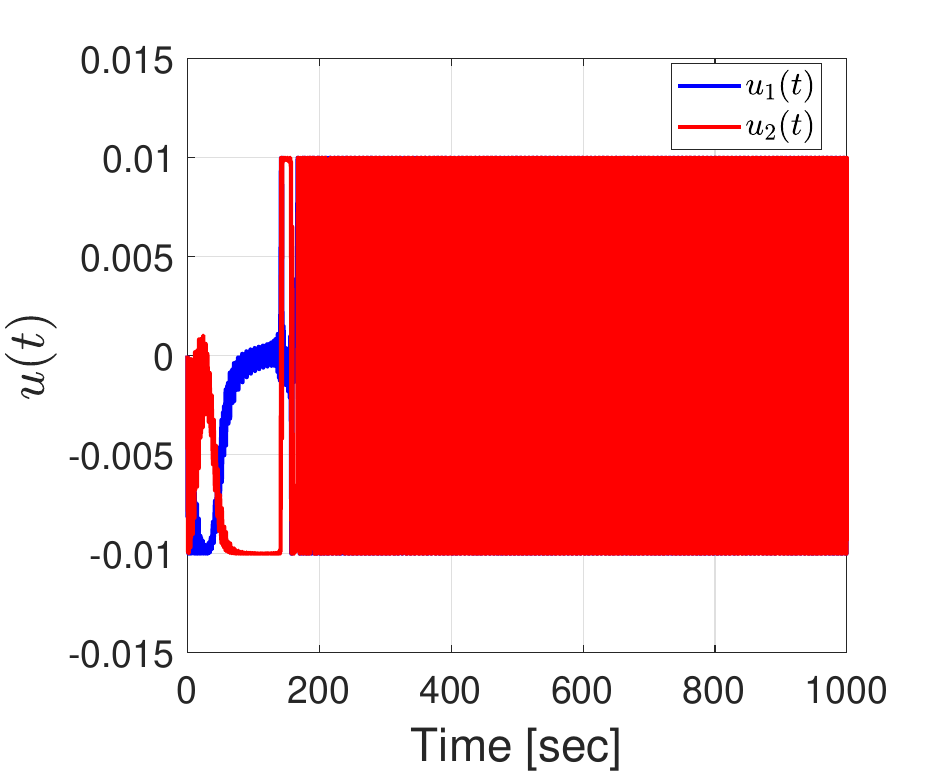}}
	\subfigure[Output $y(t)$. \label{fig:y_uvc_newton}]{\includegraphics[width=5.5cm]{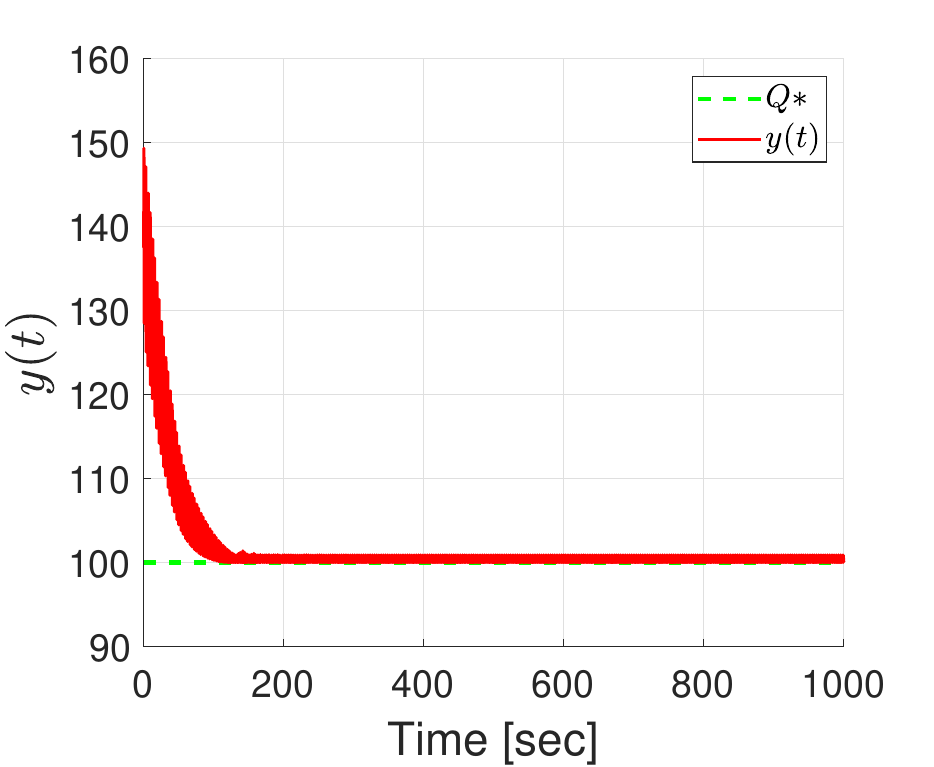}}
		\\
		\subfigure[Gradient estimate $\hat{G}(t)$. \label{fig:ghat_uvc_newton}]{\includegraphics[width=5.5cm]{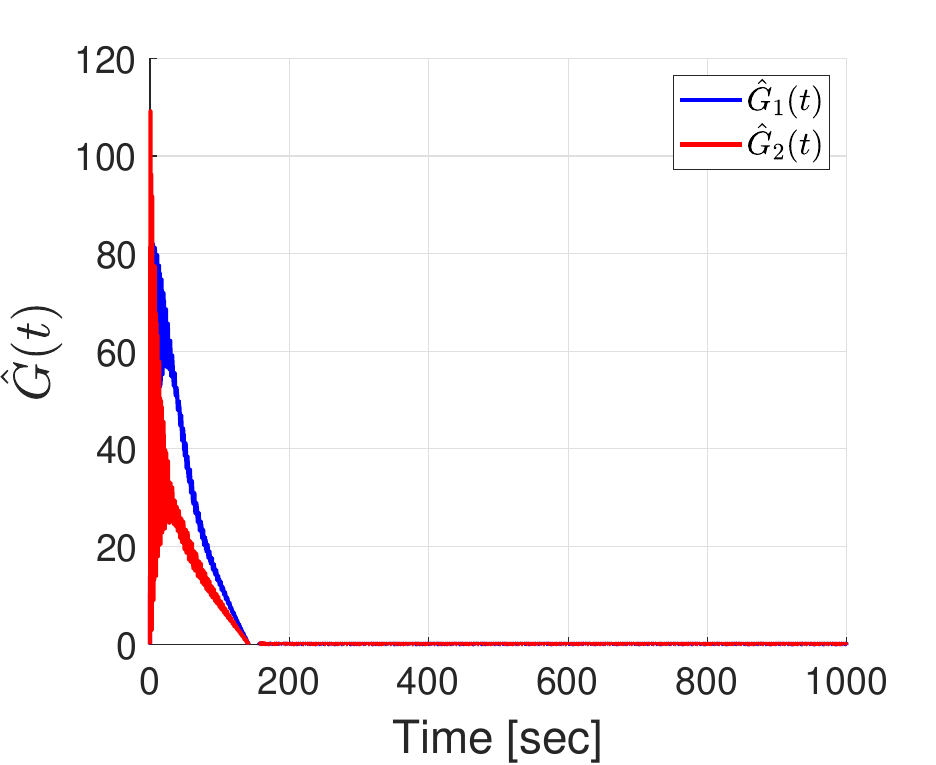}}
		\subfigure[Input $\theta(t)$. \label{fig:theta_uvc_newton}]{\includegraphics[width=5.5cm]{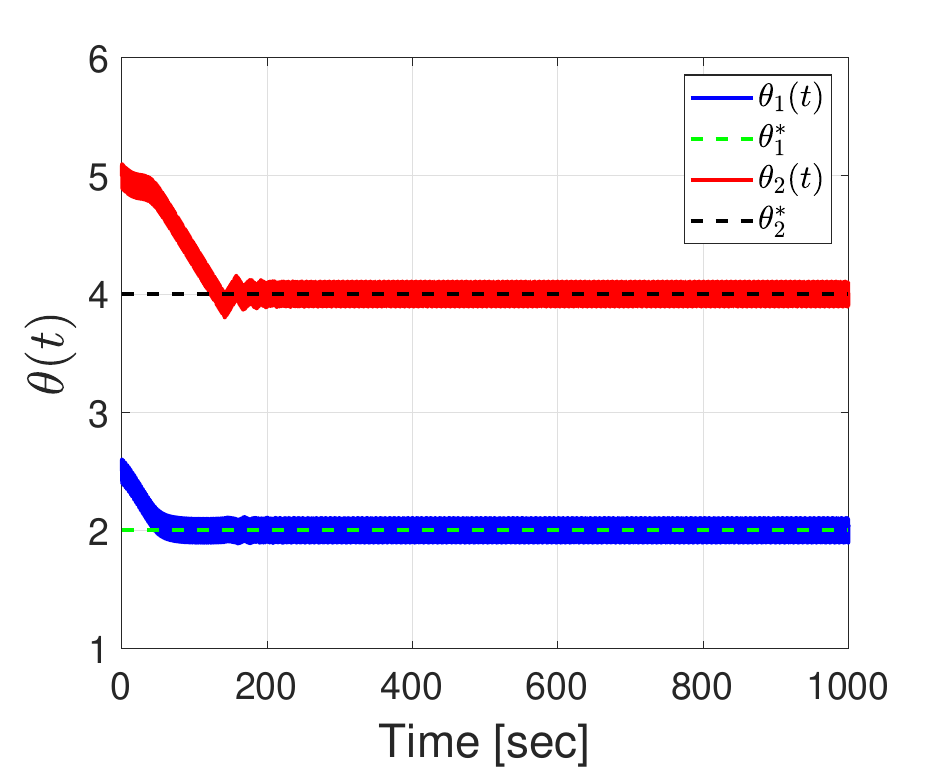}}
    \subfigure[Inverse Hessian $\Gamma^{-1}(t)$] 
    {\includegraphics[width=6.5cm]{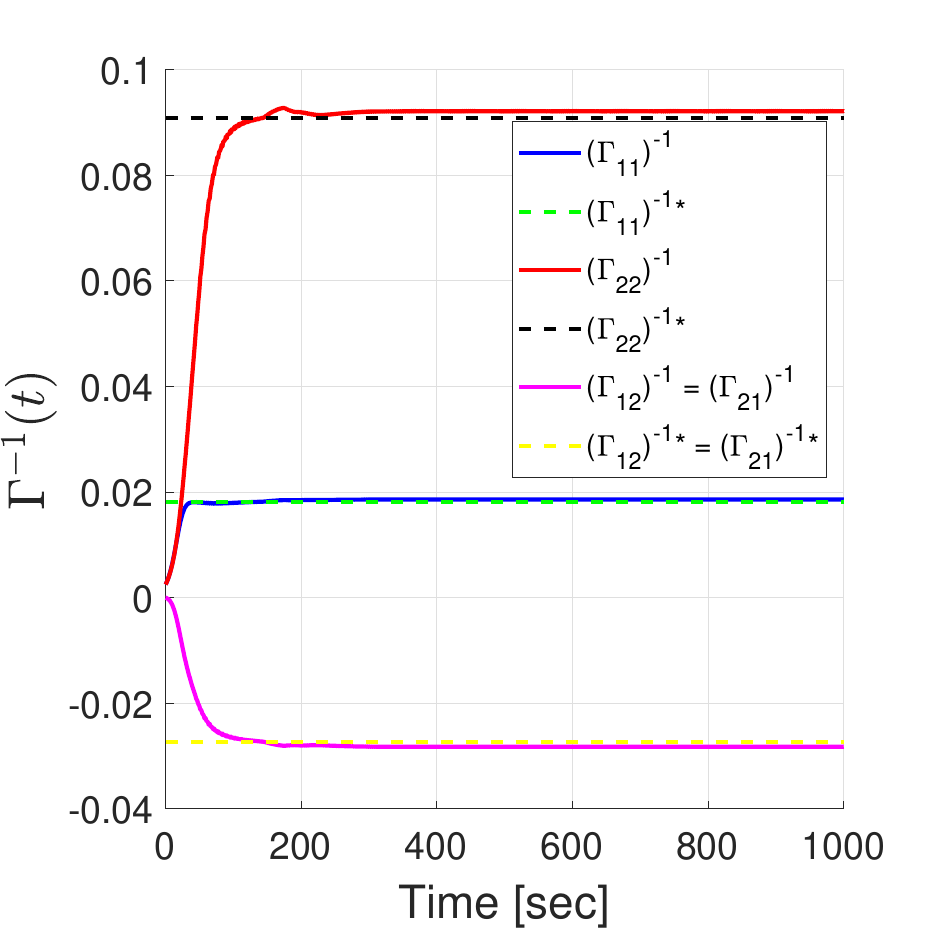}} \label{fig:inv_hessian}   
    \caption{The simulation variables are: (a) control signal $u(t)$, (b) output signal $y(t)$, (c) gradient estimate $\hat{G}(t)$, (d) input signal $\theta(t)$, and (e) estimate of the inverse of the Hessian $\Gamma^{-1}(t)$. }
    \label{fig:uvc_newton}
\end{figure*}

\newpage
\section{Conclusions}

This paper has presented a novel approach to attain locally stable convergence to the extremum point of multi-input quadratic maps by integrating sliding mode elements to extremum seeking based on sinusoidal periodic perturbations.
A comparative analysis between Gradient-based and Newton-based extremum seeking control implementations, both incorporating unit vector control, is completely derived. In both approaches, a discontinuous control law is applied, providing robustness and enforcing finite-time convergence of the system trajectories. Indeed, finite-time stability for the average closed-loop system is proved rather than exponential stability, generally arrived at classical extremum seeking with proportional control laws. 

Simulation results illustrate that the Gradient-based method exhibits a slower convergence rate and a longer transient stage before the sliding mode is established. In contrast, the Newton-based method achieves faster convergence. Although, UVC guarantees finite-time convergence, the dependency of the Gradient-based approach on the unknown Hessian significantly limits its convergence speed. As in the classical-exponential configurations, the Newton-based method overcomes this drawback through a dynamic Riccati filter, which estimates the inverse of the Hessian and enhances performance, especially in the transient phase.

The results presented in this paper offer a sliding mode control alternative to the continuous yet non-smooth equilibrium seeking algorithms with fixed-time convergence proposed in \cite{PK:2021,PKB:2023}. Our approach employs unit-vector functions, applied to estimates of the gradient, and leverages Plotnikov's averaging theorem for differential inclusions \cite{P:1979} to analyze the resulting discontinuous dynamics. This yields finite-time convergence, with a settling time that depends on the initial conditions. While the fixed-time convergence in \cite{PK:2021,PKB:2023} may appear stronger---since it ensures convergence within a time bound independent of the initial state---our method provides a complementary perspective based on sliding mode control for extremum seeking.

\textcolor{black}{
Future investigation lies in the expansion of proposed design and analysis for different control problems with unknown control direction (or unknown Hessian signs) and pursuing infinite-dimensional multi-agent optimization via Nash equilibrium seeking, as considered in references \cite{Oliveira_controldirection,TAC_DO:2024,TRO_book2022,ORKT:2021a}, rather than only scalar or multiparameter real-time optimization without delays and PDEs. Adaptive or integral sliding mode control \cite{Oliveira-Cunha-Hsu-Springer-2017,Rodrigues-Oliveira-IJC-2018,SCP:2019} are also fruitful scenarios for further developments with respect to unit vector control.}






\begin{thebibliography}{99}

\bibitem{Aminde_SMC1}
N.~O. Aminde, T.~R. Oliveira, and L.~Hsu, Global output-feedback extremum
  seeking control via monitoring functions, in \emph{52nd IEEE Conference on
  Decision and Control (CDC)}, pp. 1031--1036, 2013.

\bibitem{Aminde_SMC2}
N.~O. Aminde, T.~R. Oliveira, and L.~Hsu, Multivariable extremum seeking control via cyclic search and
  monitoring function, \emph{International Journal of Adaptive Control and
  Signal Processing}, vol.~35, pp. 1217--1232, 2021.
%
\textcolor{black}{
\bibitem{NOH:2024a}      
N. O. Aminde, T. R. Oliveira and L. Hsu, Global Output-Feedback Extremum Seeking Control with Source Seeking Experiments,\em{ International Journal of Robust and Nonlinear Control, vol. 35, pp. 5156--5171}, 2025.}
%
\bibitem{A:1957}
T.~Apostol, 
\newblock {\em Mathematical Analisys - A Modern Approach to Advanced Calculus}.
\newblock Addison-Wesley Publishing Company, 1957.  

\bibitem{AK:2003}
K.~B.~Ariyur and M.~Krsti{\' c}.
\newblock {\em Real-Time Optimization by Extremum-Seeking Control}.
\newblock Wiley, Canada, 2003.  

\bibitem{B:1993}
S. V. Baida, Unit sliding mode control in continuous and discrete-time systems, {\em International Journal of Control}, 57:5, 1125-1132, 1993.
%
\textcolor{black}{
\bibitem{M:2016}
M. Benosman,  Multi-parametric extremum seeking-based iterative feedback gains tuning for nonlinear control, \em{International Journal of Robust Nonlinear Control},  vol. 26, pp. 4035--4055, 2016.}
%
\textcolor{black}{
\bibitem{TAC_DO:2024}
A. Dibo and T. R. Oliveira,  Extremum seeking feedback under unknown {Hessian} signs, \em{IEEE Transactions on Automatic Control},  vol. 69, pp. 2383--2390, 2024.}
%
\bibitem{GKN:2012}
A.~Ghaffari, M.~Krsti{\'c}, and D.~Ne{\u s}ic, 
\newblock Multivariable {N}ewton-based extremum seeking.
\newblock {\em Automatica}, 48:1759--1767, 2012.

\bibitem{GKS:2014}
A. Ghaffari, M. Krstić and S. Seshagiri, Power Optimization for Photovoltaic Microconverters Using Multivariable Newton-Based Extremum Seeking, {\em IEEE Transactions on Control Systems Technology}, vol. 22, no. 6, pp. 2141-2149, Nov. 2014.

\bibitem{GO:2024}
A. Ghaffari, T. R. Oliveira, Second-Order Newton-Based Extremum Seeking for Multivariable Static Maps, {\em arXiv preprint arXiv:2404.01103}, 2024.
%
\textcolor{black}{
\bibitem{GMD:2014}
M. Guay, E. Moshksar, and D. Dochain, A constrained extremum-seeking control approach, \em{International Journal of Robust and Nonlinear Control}, vol. 25, pp. 3132--3153, 2014.}
%
\bibitem{GUO:2025}
H. Guo, J. Lin, K. Hong, G. Xu, Y. Sun, M. Su, Y. Liu, Multivariable Newton-Based Extremum Seeking Control for 13.56 MHz RF Impedance Matcher, {\em IEEE Transactions on Industrial Electronics}, 2025.

\bibitem{K:2002}
H.~K.~Khalil, 
\newblock {\em Nonlinear Systems}.
\newblock Prentice Hall, Upper Saddle River, New Jersey, 2002.

\bibitem{KW:2000}
M.~Krsti{\' c} and H.-H. Wang, Stability of extremum seeking feedback for
  general nonlinear dynamic systems, \emph{Automatica}, vol.~36, pp.
  595--601, 2000.

\bibitem{MMB:2010}
W. H. Moase, C. Manzie and M. J. Brear, Newton-like extremum-seeking for the control of thermoacoustic instability. {\em IEEE Transactions on Automatic Control}, 55(9), 2094–2105, 2010.  

\bibitem{M:2011}
J.~A. Moreno, Lyapunov approach for analysis and design of second order
  sliding mode algorithms, in \emph{Sliding Modes after the First Decade of
  the 21st Century}, ser. Lecture Notes in Control and Information Sciences,
  L.~Fridman, J.~Moreno, and R.~Iriarte, Eds.\hskip 1em plus 0.5em minus
  0.4em\relax Berlin, Heidelberg: Springer, 2011, vol. 412, pp. 113--149.
  [Online]. Available: \url{https://doi.org/10.1007/978-3-642-22164-4_4}

\bibitem{MO:2012}
J.~A. Moreno and M.~Osorio, Strict lyapunov functions for the super-twisting
  algorithm, \emph{IEEE Transactions on Automatic Control}, vol.~57, no.~4,
  pp. 1035--1040, 2012.

 \bibitem{TRO:2019}
T. R. Oliveira, L. R. Costa, A. V. Pino and P. Paz, Extremum Seeking-based Adaptive {PID} Control applied to Neuromuscular Electrical Stimulation, {\em Annals of the Brazilian Academy of Sciences}, 2019.
%
\textcolor{black}{
\bibitem{Oliveira-Cunha-Hsu-Springer-2017}
T.~R. Oliveira, J.~P.~V.~S. Cunha and L. Hsu,   
\newblock Adaptive sliding mode control based on the extended equivalent control concept for disturbances with unknown bounds,
{\em Advances in Variable Structure Systems and Sliding Mode Control---Theory and Applications}, Springer International Publishing, pp. 149--163, 2017.}
%
\bibitem{OFH:2023}
T. R. Oliveira, L. Fridman, and L. Hsu, {\em Sliding-Mode Control and Variable-Structure Systems: The State of the Art}, Springer, 2023.

\bibitem{Oliveira_SMC1}
T.~R. Oliveira, L.~Hsu, and A.~J. Peixoto, Output-feedback global tracking
  for unknown control direction plants with application to extremum-seeking
  control, \emph{Automatica}, vol.~47, pp. 2029--2038, 2011.
%
\textcolor{black}{
\bibitem{TRO_book2022}
T.R. Oliveira and M. Krsti{\' c}, Extremum Seeking through Delays and PDEs,
\em{Society for Industrial and Applied Mathematics}, 2022.}
%
\bibitem{OKT:2017}
T. R. Oliveira, M. Krstić and D. Tsubakino, Extremum Seeking for Static Maps With Delays, {\em IEEE Transactions on Automatic Control}, vol. 62, no. 4, pp. 1911-1926, April 2017.  
%
\bibitem{Oliveira_SMC2}
T.~R. Oliveira, A.~J. Peixoto, and L.~Hsu, Global real-time optimization by
  output-feedback extremum-seeking control with sliding modes, \emph{Journal
  of the Franklin Institute}, vol. 349, pp. 1397--1415, 2012.
%
\textcolor{black}{
\bibitem{Oliveira_controldirection}
T. R. Oliveira,  A. J. Peixoto, E. V. L. Nunes and L. Hsu, 
\newblock Control of uncertain nonlinear systems with arbitrary relative degree and unknown control direction using sliding modes, {\em Int. J. Adapt. Control Signal Process.} vol. 21, pp.~692--707, 2007.}
%
\textcolor{black}{
\bibitem{ORKT:2021a}
T. R. Oliveira, V. H. P. Rodrigues,  M. Krsti{\' c},  T. Basar,  Nash Equilibrium Seeking in Quadratic Noncooperative Games Under Two Delayed Information-Sharing Schemes, \em{ Journal of Optimization Theory and Applications}, vol. 191, pp. 700--735, 2021.}
%
\bibitem{TRoux:2021dec}
T. R. Oliveira, V. H. P. Rodrigues, M. Krsti{\' c} and T. Ba{\c s}ar, Nash Equilibrium Seeking in Heterogeneous Noncooperative Games with Players Acting Through Heat PDE Dynamics and Delays, {\em 60th IEEE Conference on Decision and Control (CDC)}, 2021.

\bibitem{Ozguner_SMC}
Y.~Pan, {\"U}.~{\"O}zg{\"u}ner, and T.~Acarman, Stability and performance
  improvement of extremum seeking control with sliding mode, 
  \emph{International Journal of Control}, vol.~76, pp. 968--985, 2003.

\bibitem{TRO:2020}
P. Paz, T. R. Oliveira, A. V, Pino and A. P. Fontana, Model-Free Neuromuscular Electrical Stimulation by Stochastic Extremum Seeking, {\em IEEE Transactions on Control Systems Technology}, 2020.

\bibitem{P:1979}
V.~A.~Plotnikov, 
\newblock Averaging of differential inclusions.
\newblock {\em Ukrainian Mathematical Journal}, 31:454--457, 1980.

\bibitem{PK:2021}
J. Poveda and M. Krstic, Non-smooth extremum seeking with user-prescribed fixed-time convergence, \emph{IEEE Trans. Automat. Contr.},
  vol.~66, pp. 6156--6163, 2021.

\bibitem{PKB:2023}
J.~I. Poveda, M.~Krsti{\' c}, and T.~Ba\c{s}ar, Fixed-time {Nash} equilibrium
  seeking in time-varying networks, \emph{IEEE Trans. Automat. Contr.},
  vol.~68, pp. 1954--1969, 2023.
%
\textcolor{black}{
\bibitem{Rodrigues-Oliveira-IJC-2018}
V.~H.~P. Rodrigues  and  T. R.  Oliveira,   
\newblock Global adaptive {HOSM} differentiators via monitoring functions and hybrid state-norm observers for output feedback,
{\em International Journal of Control} vol. 91, pp.~2060--2072, 2018.}
%
\bibitem{ROHDK:2025}
V. H. P. Rodrigues, T. R. Oliveira, L. Hsu, M. Diagne, and M. Krstic, Event-triggered and periodic event-triggered extremum seeking control, {\em Automatica}, vol. 174, p. 112161, 2025.

\bibitem{ROKB:2024a}
V. H. P. Rodrigues, T. R. Oliveira, M. Krsti{\' c}, and T. Ba{\c s}ar, Nash Equilibrium Seeking for Noncooperative Duopoly Games via Event-triggered Control, {\em arXiv preprint arXiv:2404.07287}, 2024.

\bibitem{ROKB:2024b}
V. H. P. Rodrigues, T. Roux Oliveira, M. Krsti{\' c} and T. Ba{\c s}ar, Sliding-Mode Nash Equilibrium Seeking for a Quadratic Duopoly Game, 2024 17th International Workshop on Variable Structure Systems (VSS), Abu Dhabi, United Arab Emirates, 2024, pp. 99-106, doi: 10.1109/VSS61690.2024.10753

\bibitem{ROKT:2025}
V. H. P. Rodrigues, T. R. Oliveira, M. Krsti{\' c}, and P. Tabuada, Event-Triggered Newton-Based Extremum Seeking Control, {\em arXiv preprint arXiv:2502.00930}, 2025.

\bibitem{K:2014}
M.~Krsti{\'c}.
\newblock Extremum seeking control, 
\newblock In J.~Baillieul and T.~Samad, editors, {\em Encyclopedia of Systems and Control}, volume~1, pages 413--416. Springer, London, 1 edition, 2014.
%
\textcolor{black}{
\bibitem{SCP:2019}
C. U. Solis, J. B. Clempner and A. S. Poznyak, Extremum seeking by a dynamic plant using mixed integral sliding mode controller with synchronous detection gradient estimation, \em{International Journal of Robust and Nonlinear Control}, vol. 29, pp. 702--714, 2019.}
%
\bibitem{S:2008}
J.~M.~Steele,  
\newblock {\em The Cauchy-Schwartz Master Class: an introduction to the art of mathematical inequalities}.
\newblock Cambridge University Press, New York, 2008.
%
\textcolor{black}{
\bibitem{SK:2024}
R. Suttner and M.~Krsti{\' c}, Attitude optimization by extremum seeking for rigid bodies actuated by internal rotors only, \em{International Journal of Robust Nonlinear Control}, vol. 34, pp. 2384--2404, 2024.}
%
\bibitem{TK:2023}
V.~Todorovski and M.~Krsti{\' c}, Practical prescribed-time seeking of a
  repulsive source by unicycle angular velocity tuning, \emph{Automatica},
  vol.~68, p. 111069, 2023.

\bibitem{YK:2022}
C.~T. Yilmaz and M.~Krsti{\' c}, Prescribed-time extremum seeking with chirpy
  probing for {PDEs--Part I: Delay}, \emph{IEEE American Control Conference},
  pp. 1000--1005, 2022.

\bibitem{zhang2012extremum}
C. Zhang and R. Ord{\'o}{\~n}ez, {\em Extremum-Seeking Control and Applications: A Numerical Optimization-Based Approach}, Springer, New York, NY, 2012.

\end{thebibliography}


\end{document}